\documentclass[11pt,english]{article}
\usepackage[T1]{fontenc}
\usepackage[latin9]{inputenc}
\usepackage[a4paper]{geometry}
\usepackage{enumitem}
\usepackage{amsmath}
\usepackage{amsthm}
\usepackage{amssymb}

\makeatletter
\numberwithin{equation}{section}
\theoremstyle{plain}
\newtheorem{thm}{\protect\theoremname}[section]
\theoremstyle{plain}
\newtheorem{conjecture}[thm]{\protect\conjecturename}
\theoremstyle{plain}
\newtheorem{cor}[thm]{\protect\corollaryname}
\theoremstyle{plain}
\newtheorem*{thm*}{\protect\theoremname}
\theoremstyle{plain}
\newtheorem{prop}[thm]{\protect\propositionname}
\theoremstyle{definition}
\newtheorem{problem}[thm]{\protect\problemname}
\theoremstyle{plain}
\newtheorem{lem}[thm]{\protect\lemmaname}
\theoremstyle{remark}
\newtheorem*{rem*}{\protect\remarkname}

\newcommand{\hide}[1]{}

\DeclareMathOperator{\diam}{diam}
\DeclareMathOperator{\supp}{supp}

\DeclareMathOperator{\cov}{cov}

\DeclareMathOperator{\diag}{diag}

\DeclareMathOperator{\dense}{Dense}

\makeatother

\usepackage{babel}
\providecommand{\conjecturename}{Conjecture}
\providecommand{\corollaryname}{Corollary}
\providecommand{\lemmaname}{Lemma}
\providecommand{\problemname}{Problem}
\providecommand{\propositionname}{Proposition}
\providecommand{\remarkname}{Remark}
\providecommand{\theoremname}{Theorem}

\begin{document}
\date{}
\title{Dense $\times b$-orbits from $\times a$-invariant sets and measures,
in zero entropy }
\author{Michael Hochman\thanks{Research supported by ISF research grant 3056/21. The author wishes
to thank the Department of Mathematics of the University of Warwick
for its hospitality during the 2024/5 academic year.}}
\maketitle
\begin{abstract}
Let $a,b$ be multiplicatively independent integers. We prove that
if $\mu$ is a zero-entropy non-atomic $\times a$-invariant measure
on $[0,1]$, then $\mu$-a.e.~$x$ has dense $\times b$-orbit. Related
statements follow for many closed $\times a$-invariant sets, and
in particular we conclude the existence of irrational numbers with
minimal complexity in base $a$ and maximal complexity in base $b$.
We extend~these results to commuting endomorphisms of the torus under
an irreducibility and hyperbolicity condition. 
\end{abstract}

\section{\label{sec:Introduction}Introduction}

\subsection{Background}

One of the most perplexing problems of metric number theory is to
explain how constraints on the base-$a$ expansion of a real number
affect its base-$b$ expansion, where $a,b\geq2$ are multiplicatively
independent integers  (i.e.~$a^{m}\neq b^{n}$ for all $m,n\in\mathbb{N}$).
By ``constraints'' we will understand either that $x$ is restricted
to a closed set $X\subseteq[0,1]$ invariant under the map $T_{a}:x\mapsto ax\bmod1$
(we call such a set an $a$-set), or that it is drawn at random from
a $T_{a}$-invariant Borel probability measure (an $a$-measure).
The first of these amounts to restricting which finite blocks of $a$-digits
can appear in the expansion, and the second to prescribing their frequencies.
We gauge the effect on the base-$b$ expansion similarly, in terms
of the orbit $O_{b}(x)=\{T_{b}^{n}x\}_{n\geq0}$ and the $b$-set
that is its closure. The orbit may be dense (all blocks of $b$-digits
appear in $x$), in which case $x$ is said to be $b$-transitive;
or the orbit may uniformly distribute in $[0,1]$ (all blocks of $n$
digits appear in $x$ with frequency $b^{-n}$), in which case $x$
is called $b$-normal. A more subtle measure for the size of the orbit
closure $Y=\overline{O_{b}(x)}$ is its Hausdorff dimension $\dim Y$,
which for $b$-sets coincides with its Minkowski (box) dimension $\dim_{B}Y$,
and corresponds to the exponential growth rate of the number of base-$b$
blocks of length $n$ in the $b$-expansion of $x$, or, up to a constant,
with the topological entropy of $T_{b}$ acting on $Y$.

The central tenet surrounding this question is that constraining $x$
in base $a$ leaves the $b$-expansion free to be, and perhaps even
forces it to be, highly complex. This was made precise around 1970
by Furstenberg, who proposed the following conjecture \cite{furstenberg1970intersections}: 

\begin{conjecture}
If $a,b\geq2$ are multiplicatively independent integers, then 
\begin{equation}
\dim\overline{O_{a}(x)}+\dim\overline{O_{b}(x)}\geq1\qquad\text{for all }x\in[0,1]\setminus\mathbb{Q}\label{eq:transversality}
\end{equation}
Equivalently, if $X\subseteq[0,1]$ is an $a$-set, then every $x\in X\setminus\mathbb{Q}$
satisfies $\dim\overline{O_{b}(x)}\geq1-\dim X$.
\end{conjecture}

This conjecture is very much open, and seems well beyond the reach
of current methods. But the work of many authors, starting with Cassels
and Schmidt \cite{Cassels1959,Schmidt1960}, supports it by showing
that certain $a$-sets of positive dimension, especially those satisfying
specification-type properties, contain many points $x$ satisfying
$\dim\overline{O_{b}(x)}\geq1-\dim X$. We mention two results which
summarize the state of the art. 

The first is Host's theorem, which states that if $\mu$ is an ergodic
$a$-measure of positive entropy, then $\mu$-a.e.~point is $b$-normal,
and hence also $b$-transitive \cite{Host95,Lindenstrauss2001b,HochmanShmerkin2015-equidistribution-from-fractal-measures}.
To apply this we merely note that if $X$ is an $a$-set and $\dim X>0$,
then the topological entropy of $X$ is positive, so we can find an
ergodic $a$-measure $\mu$ on $X$ with positive entropy; whence
Host's theorem tells us that the measure is supported on $b$-normal
(and $b$-transitive) points in $X$. This result is closely related
to Furstenberg's $\times2,\times3$-problem, which we return to later,
and like all other results about on the $\times2,\times3$ problem,
Host's result is not known when the entropy is zero.

The other relevant result is the groundbreaking work of Shmerkin and
Wu, who realized part of Furstenberg's program towards the conjecture
\cite{Shmerkin2019,Wu2019}: They showed that for every $a$-set $X$
and $b$-set $Y$, 
\begin{equation}
\dim X\cap Y\leq\max\{0,\dim X+\dim Y-1\}.\label{eq:intersections}
\end{equation}
From this it can be shown that $\dim\overline{O_{b}(x)}\geq1-\dim X$
outside a set of $x\in X$ of dimension $0$. Thus, if $\dim X>0$,
then ``most'' $x\in X$ will satisfy $\dim\overline{O_{b}(x)}\geq1-\dim X$. 

Unfortunately, none of the results above, nor, to our knowledge, any
other existing work, apply when $\dim X=0$ or $h_{\mu}(T_{a})=0$.
We point out that when $\dim X=0$, Furstenberg's conjecture predicts
that the $b$-orbit of every $x\in X\setminus\mathbb{Q}$ should be
dense, since its closure is a full-dimensional $b$-set.

\subsection{Main results}

Our contribution to this matter is the theorem below, which establishes
the existence of $b$-transitive points for many $a$-sets and $a$-measures.
The novelty is that it covers the zero-entropy case as well; when
the entropy is positive, Host's theorem gives a much stronger conclusion. 
\begin{thm}
\label{thm:main-1-d}Let $a,b\geq2$ be multiplicatively independent
integers. Let $\mu\in\mathcal{P}([0,1])$ be a non-atomic, $T_{a}$-invariant
and ergodic probability measure. Then $\mu$-a.e.~$x$ has dense
orbit under $T_{b}$. 
\end{thm}

If one could upgrade $b$-transitivity to $b$-normality, it would
resolve the $\times2,\times3$-problem, but our methods are far too
crude to address this (see also remarks at the end of Section \ref{subsec:Relation-with-rigidity}).
The proof of Theorem \ref{thm:main-1-d} is closer to Furstenberg's
original argument for density of orbits under the joint action, combined
with some new analysis.

There is an obvious application  of the theorem to $a$-sets. The
hypotheses of the corollary below are automatic for e.g.~minimal
$a$-sets (where every orbit is dense):
\begin{cor}
\label{cor:sets-supporting-non-atomic-measures}Let $\text{\ensuremath{a,b\geq2}}$
be multiplicatively independent integers. Let $X\subseteq[0,1]$ be
an $a$-set that admits a fully supported non-atomic $a$-measure.
Then $X$ contains a dense $G_{\delta}$-set of points with dense
$b$-orbit.
\end{cor}

This follows because, when $X$ admits a fully supported $a$-measure,
Theorem \ref{thm:main-1-d} implies that the set of $b$-transitive
points is dense in $X$; and it is a standard fact that the set of
such points is a $G_{\delta}$-set. With regard to the hypothesis,
we note that there do exist uncountable $a$-sets that support only
atomic $a$-measures. According to the conjecture, such sets should
contain $b$-transitive points, but this does not follow from our
methods and remains open.

Corollary \ref{cor:sets-supporting-non-atomic-measures} has the following
consequence in terms of complexity of digital expansions. Recall that
for $x\in[0,1]$, the base-$a$ complexity function $c_{a,n}(x)$
is defined for $n\in\mathbb{N}$ to be the number of distinct words
 $w\in\{0,\ldots,a-1\}^{n}$ that appear consecutively in the base-$a$
expansion of $x$. If $x$ is irrational, then $c_{a,n}(x)\geq n+1$,
and there exist numbers that realize this minimal growth rate. The
maximal possible growth is $c_{a,n}(x)=a^{n}$, and is realized by
$a$-transitive points. By applying Theorem \ref{thm:main-1-d} to
an $a$-measure supported on points of minimal complexity (a Sturmian
measure in the base-$a$ coding), we get
\begin{cor}
Given $2\leq a\in\mathbb{N}$ there exist $x\in[0,1]$ with minimal
complexity in base $a$ and maximal complexity in every base $b$
that is independent of $a$.
\end{cor}

We can also re-state Corollary \ref{cor:sets-supporting-non-atomic-measures}
in terms of intersections. It is far less quantitative than (\ref{eq:intersections}),
but gives non-trivial information when the dimension is zero.
\begin{cor}
For $a,b,X$ as in Corollary \ref{cor:sets-supporting-non-atomic-measures},
and any non-trivial $b$-set $Y$, the intersection $X\cap Y$ has
empty interior in $X$.
\end{cor}

\subsection{\label{subsec:Relation-with-rigidity}Relation with rigidity and
the role of commutation}

Furstenberg's rigidity theorem, proved in \cite{Furstenberg67}, is
the following statement:
\begin{thm*}
[Furstenberg] If $2\leq a,b\in\mathbb{N}$ are multiplicatively independent,
and if $X\subseteq[0,1]$ is both an $a$-set and a $b$-set (i.e.~it
is closed and \emph{jointly} invariant under  $T_{a},T_{b}$), then
either $X$ is finite or $X=[0,1]$.
\end{thm*}
It turns out that when $X$ is a minimal $a$-set (i.e., no proper
closed subset of $X$ is $T_{a}$-invariant), the conclusion of Corollary
\ref{cor:sets-supporting-non-atomic-measures} follows by a short
argument from Furstenberg's rigidity theorem. The argument is quite
general, and we obtain the statement below, from which one recovers
Corollary \ref{cor:sets-supporting-non-atomic-measures} in the minimal
case by taking $S=T_{a}$, $T=T_{b}$ and $\mathcal{F}$ the family
of finite subsystems of $X$.

To fix terminology, given a continuous map $S:X\rightarrow X$ of
a metric space, a subsystem is a closed set $Y\subseteq X$ satisfying
$SY\subseteq Y$. A point $x\in Y$ is transitive (for $Y$) if the
orbit $O_{S}(x)=\{S^{n}x\}_{n\geq0}$ is dense in $Y$, the set  $Y$
is transitive if it contains a transitive point, and $Y$ is minimal
if all points in $Y$ are transitive. When more than one map is involved
we shall say $S$-transitive point, $S$-subsystem, etc.
\begin{prop}
\label{prop:minimal-sets}Let $X$ be a compact metric space, let
$S,T:X\rightarrow X$ be continuous and onto maps with $ST=TS$ and
$T$ transitive.

Let $\mathcal{F}$ be a family of $S$-subsystems of $X$ which is
closed to taking subsystems (i.e.~if $Y\in\mathcal{F}$ and $Z$
is a subsystem of $Y$ then $Z\in\mathcal{F}$). Assume further that 
\begin{quote}
({*}) The only jointly $S,T$-invariant sets are elements of $\mathcal{F}$
and $X$ itself.
\end{quote}
Then every minimal $S$-subsystem $Y\subseteq X$ that does not belong
to $\mathcal{F}$ contains a dense $G_{\delta}$-set of $T$-transitive
points.
\end{prop}

The assumption that $T$ is transitive is necessary. To see this,
let $X=\mathbb{R}^{2}/\mathbb{Z}^{2}$, let $\alpha\in\mathbb{R}\setminus\mathbb{Q}$,
and $S,T$ act on $X$ by translation by $(\alpha,0)$ and $(0,\alpha)$,
respectively. Then the joint action of $S,T$ is minimal, so, taking
$\mathcal{F}=\emptyset$, all assumptions of the proposition hold
except transitivity of $T$. At the same time the conclusion fails:
the subset $(\mathbb{R}/\mathbb{Z})\times\{0\}$ is an $S$-minimal
subset containing no $T$-transitive points.

We do not know whether the commutation assumption can be relaxed,
but it plays essential role in all of our proofs. Nevertheless there
are many interesting cases where commutativity fails but rigidity
holds. For $t\in[0,1]$ let $R_{t}x=x+t\bmod1$ and consider the map
$T_{a,t}=R_{t}^{-1}T_{a}R_{t}$ of $[0,1]$. In \cite{Hochman2018-smooth-symmetries-of-xa-invariant-sets}
we proved rigidity of $T_{a,t},T_{b}$, i.e.~that there are no infinite
jointly $T_{a,t},T_{b}$-invariant minimal sets. We do not know if
the analog of Proposition \ref{prop:minimal-sets} holds:
\begin{problem}
For $a,b$ multiplicatively independent and $t\in[0,1]$, if $X\subseteq[0,1]$
is closed, infinite and minimal under $T_{a,t}$ as above, does it
contain $b$-transitive points? 
\end{problem}

Another interesting case is the pair $T_{a},G$ where $G(x)=\frac{1}{x}\bmod1$
is the Gauss map. For many base-$a$ Cantor-type sets $X$ are known
to contain badly approximable numbers, and in fact, points with dense
$G$-orbit \cite{SimmonsWeiss2019}. No such statement is known when
$\dim X=0$ (or, in fact, when $X$ is a more complicated positive-dimension
$a$-set): 
\begin{problem}
If $X$ is an $a$-set, does it contain points with dense orbit under
the Gauss map?
\end{problem}

For a related problem see Problem 10.50 in \cite{Bugeaud_2012}.

We conclude with a remark about measure rigidity. The measure analog
of Furstenberg's rigidity theorem is his so-called $\times2,\times3$-problem:
It asks if the only non-atomic, jointly-$T_{a},T_{b}$-invariant probability
measure on $[0,1]$ is Lebesgue. The problem remains open; although
an affirmative answer was given by Rudolph and Johnson for measures
with positive $T_{a}$-entropy \cite{Rudolph90,Johnson92}, nothing
is known in entropy zero.

In view of Proposition \ref{prop:minimal-sets}, the question arises
whether Theorem \ref{thm:main-1-d} would follow similarly from a
positive answer to the $\times2,\times3$-problem. As far as we can
see, it does not, because such a derivation would most likely need
to rule out the possibility that there exists a zero entropy $a$-measures
whose typical points do not equidistribute under $T_{b}$, or equidistribute
for a non-ergodic or an atomic measure.

\subsection{Commuting endomorphisms of the torus}

Furstenberg's rigidity theorem on jointly $T_{a},T_{b}$-invariant
sets was extended by Berend to commuting semigroups of toral endomorphisms
\cite{Berend1983}. Proposition \ref{prop:minimal-sets} applies in
this setting under the same assumptions as Berend's theorem. Our last
result is a generalization of Theorem \ref{thm:main-1-d}, using version
of Host's theorem that we recently established for total endomorphisms
\cite{Hochman2025a}. 

Write $M_{d}(\mathbb{Z})$ for the semigroup of $d\times d$ integer
matrices. Let $T_{A}$ denote the induced action of $A\in M_{d}(\mathbb{Z})$
on $\mathbb{T}^{d}=\mathbb{R}^{d}/\mathbb{Z}^{d}.$ A commuting pair
$A,B\in M_{d}(\mathbb{Z})$ acts totally irreducibly on $\mathbb{T}^{d}$
if there does not exist an infinite closed proper subgroup $G\leq\mathbb{T}^{d}$
such that $\{T_{A}^{m}T_{B}^{n}G\}_{m,n\in\mathbb{N}}$ is finite.
The action is higher rank if the joint action is not virtually cyclic.
Finally, recall that a matrix is expanding if all of its eigenvalues
lie outside the unit circle, and hyperbolic if none of its eigenvalues
lie on the unit circle.
\begin{thm}
\label{thm:main-multi-d}Let $A,B\in M_{d}(\mathbb{Z})$ be non-singular
matrices satisfying $AB=BA$ and with $T_{A},T_{B}$ acting totally
irreducibly on $\mathbb{T}^{d}$ and the joint action is higher rank.

Let $\mu\in\mathcal{P}(\mathbb{T}^{d})$ be a non-atomic, $T_{A}$-invariant
and ergodic probability measure. Then 
\begin{enumerate}
\item [a.] If $B$ is expanding, then the one-sided orbit $O_{B}(x)=\{T_{B}^{n}x\}_{n\geq0}$
is dense for $\mu$-a.e.~$x$.
\item [b.] If $B$ is hyperbolic and $T_{B}$ is invertible (i.e.~$\det B=\pm1$),
then the two-sided orbit $\{T_{B}^{n}x\}_{n\in\mathbb{Z}}$ is dense
for $\mu$-a.e.~$x$.
\end{enumerate}
\end{thm}

The hyperbolicity condition is necessary because one can find matrices
$A,B$ satisfying the other assumptions, but which admit an irreducible
subspace $V$ on which the action of $A,B$ is conjugated to rotations
by irrational angles. Then $V$ admits jointly invariant non-atomic
measures supported on a closed curve, and reducing modulo one, we
have a measure on $\mathbb{T}^{d}$ violating the conclusion of the
theorem. This phenomenon is the reason for a similar expansion requirement
in Berend's theorem \cite{Berend1983}.

In dimensions $d\geq2$ it also makes sense to consider pairs of non-commuting
matrices. In fact, the Host-type result we proved in \cite{Hochman2025a}
does not require $A,B$ to commute, provided that $T_{B}$ acts totally
irreducibly on its own, satisfies some expansion assumptions, and
the eigenvalues of the matrices are pairwise multiplicatively independent.
Using this one derives Theorem \ref{thm:main-multi-d} without commutation
in the positive entropy case, but it leaves the zero-entropy case
open:
\begin{problem}
Let $A,B\in M_{d}(\mathbb{Z})$. Suppose that $B$ is hyperbolic,
$T_{B}$ acts invertibly and totally irreducibly on $\mathbb{T}^{d}$,
and $a,b$ are multiplicatively independent for every eigenvalue $a$
of $A$ and $b$ of $B$. If $\mu$ is $T_{A}$-invariant and non-atomic,
does $\mu$-a.e.~$x$ have a dense $T_{B}$-orbit?
\end{problem}

\subsection{Organization of the paper}

In the next section we introduce some basic notation. We prove Proposition
\ref{prop:minimal-sets} in Section \ref{sec:Proof-of-Proposition-minimal}.
Theorem \ref{thm:main-1-d} is proved in Section \ref{sec:Proof-of-1-d-case}
and Theorem \ref{thm:main-multi-d} is proved in Section \ref{subsec:Proof-of-Theorem-multi-d}.

\subsection{AI disclosure}

This research was carried out entirely by the author during the 2024/25
academic year and submitted for publication in December 2025. The
preprint release was delayed until September 2026 and just prior to
it, ChatGPT 5.6 Terra was used to help locate errata. Resulting corrections
were done out manually. The author takes full responsibility for the
content.

\section{\label{sec:Standing-notation}Standing notation }

For a metric space $(X,d)$, a set $Y\subseteq X$ and $\varepsilon>0$
is $\varepsilon$-dense if for every $x\in X$ there exists $y\in Y$
with $d(x,y)<\varepsilon$. If $T:X\rightarrow X$ and $k\in\mathbb{N}$,
write
\begin{align*}
\dense_{T}(\varepsilon,k) & =\{x\in X\mid\{T^{n}x\}_{n=0}^{k}\text{ is \ensuremath{\varepsilon}-dense in }X\}\\
\dense_{T}(\varepsilon) & =\{x\in X\mid\{T^{n}x\}_{n=0}^{\infty}\text{ is \ensuremath{\varepsilon}-dense in }X\}\\
 & =\bigcup_{k}\dense_{T}(\varepsilon,k)
\end{align*}
These sets are open and $\dense_{T}(\varepsilon,k)\subseteq\dense_{T}(\varepsilon,m)$
whenever $k<m$. When $T=T_{b}$ for $b\in\mathbb{N}$ we abbreviate
$\dense_{b}(\varepsilon)=\dense_{T_{b}}(\varepsilon)$, etc., and
similarly for $T=T_{B}$, $B\in M_{d}(\mathbb{Z})$.

For a compact set $Y\subseteq X$ and $r>0$, the covering number
of $Y$ at scale $r$ is
\[
\cov(Y,r)=\;\text{the minimal number of \ensuremath{r}-balls needed to cover \ensuremath{Y}}
\]
We introduce the non-standard notation for ``scale-$r$ dimension'':
\[
\dim(Y,r)=\frac{\log\cov(Y,r)}{\log(1/r)}.
\]
The box dimension of $Y$ is given by $\dim_{B}Y=\lim_{r\rightarrow0}\dim(Y,r)$
when the limit exists. We recall again that it exists when $Y$ is
a $b$-sets, and $\dim Y=\dim_{B}Y$.

The Shannon entropy of a probability measure $\mu\in\mathcal{P}(\mathbb{R}^{N})$
at scale $r>0$ is defined using an auxiliary partition $\mathcal{Q}=\{Q_{i}\}$
of $\mathbb{R}^{N}$ into cubes of side $r$ (on the line, intervals
of side $r$): 
\[
H(\mu,r)=-\sum_{i}\mu(Q_{i})\log\mu(Q_{i}).
\]
The logarithm is always in base two. The choice of $\mathcal{Q}$
affects the entropy only up to an additive constant, which can be
ignored, since in our applications we always normalize by $\log(1/r)$,
making it negligible. For a sequence $\underline{x}=(x_{n})_{n=1}^{N}$
write $H(\underline{x},r)$ for the scale-$r$ entropy of $\frac{1}{N}\sum_{n=1}^{N}\delta_{x_{n}}$.

We define entropy for measures on $\mathbb{T}^{d}$ by lifting them
to $[0,1]^{d}$ and applying the previous definition.

We rely freely on standard properties of entropy and covering numbers,
for example
\begin{equation}
H(\mu,r)\leq\log\cov(\supp\mu,r)+O(1).\label{eq:covering-bounds-entropy}
\end{equation}
or that if $f:\mathbb{R}^{d}\rightarrow\mathbb{R}^{d}$ is bi-Lipschitz
with constant $C$ then 
\[
|H(\mu,r)-H(f\mu,r)|<O(|\log C|)
\]

\section{\label{sec:Proof-of-Proposition-minimal}Proof of Proposition \ref{prop:minimal-sets}}

Let $S,T:X\rightarrow X$ continuous, commuting, onto maps of a compact
metric space. Assume that $T$ acts transitively and let $\mathcal{F}$
be a family of $S$-subsystems of $X$ that is closed to taking subsystems
and satisfies
\begin{quote}
The only jointly $S,T$-invariant sets are elements of $\mathcal{F}$
and $X$ itself.
\end{quote}
Let $Y\subseteq X$ be $S$-minimal and $Y\notin\mathcal{F}$. Our
goal is to prove that $Y$ contains a dense $G_{\delta}$ set of $T$-transitive
points.
\begin{lem}
\label{lem:T-transitive-points-move-with-S}For every $\varepsilon>0$
and $N\in\mathbb{N}$ there exists $\delta>0$ such that $x\in\dense_{T}(\delta)$
implies that $x,Sx,\ldots,S^{N}x\in\dense_{T}(\varepsilon)$.
\end{lem}

\begin{proof}
Fix $\varepsilon>0$ and $N\in\mathbb{N}$, and use continuity of
$S$ to find a $\delta>0$ such that $d(x,y)<\delta$ implies $(S^{n}x,S^{n}y)<\varepsilon$
for $n=0,1,\ldots,N$.

Suppose that $x\in\dense_{T}(\delta)$ and fix $n\in\{0,\ldots,N\}$.
Given $y\in X$, choose $z\in S^{-n}y$ (this is possible since $S^{n}$
is onto). Since $x\in\dense_{T}(\delta)$ we can find $k\geq0$ such
that $d(T^{k}x,z)<\delta$. By choice of $\delta$ this implies $d(S^{n}T^{k}x,S^{n}z)<\varepsilon$,
which, since $S^{n}z=y$ and $S,T$ commute, is the same as $d(T^{k}S^{n}x,y)<\varepsilon$. 

Summarizing, for all $y\in X$ we found $k\geq0$ with $d(T^{k}S^{n}x,y)<\varepsilon$,
so $S^{n}x\in\dense_{T}(\varepsilon)$.
\end{proof}
Recall that $Y\subseteq X$ is an $S$-minimal subsystem and $Y\notin\mathcal{F}$.
For $y\in Y$ let 
\begin{align*}
Z_{y} & =\overline{\{T^{n}y\}_{n\geq0}}
\end{align*}
and set
\[
Z=\overline{\bigcup_{y\in Y}Z_{y}}
\]

\begin{lem}
\label{lem:Z-equals-X}$Z=X$.
\end{lem}

\begin{proof}
Since $S,T$ commute, $SZ_{y}=Z_{Sy}$, so $Z$ is $S$-invariant.
Alsom $TZ_{y}\subseteq Z_{y}$, so $Z$ is $T$-invariant. Thus, either
$Z=X$ or $Z\in\mathcal{F}$. The latter is ruled out since $y\in Z_{y}\subseteq Z$,
hence by $S$-invariance, $\{S^{n}y\}_{n\geq0}\subseteq Z$, and by
minimality, $Y=\overline{\{S^{n}y\}}\subseteq Z$. Since $\mathcal{F}$
is closed to taking subsystems and $Y\notin\mathcal{F}$ by assumption,
we must have $Z\notin\mathcal{F}$. 
\end{proof}
\begin{lem}
\label{lem:T-eplilon-dense-points-are-dense-in-Y}For all $\varepsilon>0$
and $r>0$, the set $\dense_{T}(\varepsilon)\cap Y$ is $r$-dense
in $Y$.
\end{lem}

\begin{proof}
Fix $\varepsilon,r>0$. By $S$-minimality of $Y$, there exists $N$
such that $\{S^{n}x\}_{n=0}^{N}\subseteq Y$ is $r$-dense in $Y$
for every $x\in X$. 

Apply Lemma \ref{lem:T-transitive-points-move-with-S} with parameters
$\varepsilon$ and $N$ to obtain $\delta>0$ as in the lemma.

Now choose a point $x_{0}\in X$ with dense $T$-orbit, as we can
do by $T$-transitivity. By continuity of $T$, every $x\in X$ sufficiently
close to $x_{0}$ will have $\delta$-dense $T$-orbit. Since $Z=X$
is the closure of the sets $Z_{y}$, there exists $y\in Y$ with $d(Z_{y},x_{0})<\delta$.
And since $Z_{y}=\overline{\{T^{n}y\}_{n\geq0}}$, there exists $n$
with $d(T^{n}y,x_{0})<\delta$, so $y\in\dense_{T}(\delta)$.

Finally, by choice of $\delta$, we have $\{S^{n}y\}_{0\leq n\leq N}\subseteq Y\cap\dense_{T}(\varepsilon)$,
and this sequence is $r$-dense in $Y$, as we claimed.
\end{proof}
We can now complete the proof of Proposition \ref{prop:minimal-sets}.
For, given $\varepsilon>0$, the fact that $\dense_{T}(\varepsilon)\cap Y$
is $r$-dense in $Y$ for every $r>0$ implies that $\dense_{T}(\varepsilon)\cap Y$
is dense in $Y$. It is also open, since $\dense_{T}(\varepsilon)$
is open in $X$, so by Baire's theorem, the intersection 
\[
D=\bigcap_{n\geq1}\dense_{T}(1/n)\cap Y
\]
is a dense $G_{\delta}$ set in $Y$ consisting of $T$-transitive
points. This is what was claimed.

\section{\label{sec:Proof-of-1-d-case}Proof of Theorem \ref{thm:main-1-d}}

\subsection{\label{subsec:preparations-1D}Preparations}

We collect a number of dynamical and analytical facts. With an eye
towards later sections we state some of them in higher dimensions.

\begin{lem}
\label{lem:small-dimension-propogates-for-a-sets}Let $2\leq a\in\mathbb{Z}$
and let $Y\subseteq[0,1]$ be an $a$-set. Suppose that $r>0$ and
$c\geq1$. Then $\dim(Y,r^{c})<(1+O(\frac{1}{c}))\dim(Y,r)$.
\end{lem}

\begin{proof}
Consider an interval $I=[\frac{p}{a^{n}},\frac{p+1}{a^{n}})\subseteq[0,1)$.
By $T_{a}^{n}Y\subseteq Y$, the set $Y\cap I$ is a scaled copy of
a subset of $Y$. Hence $\cov(Y\cap I,a^{-2n})\leq\cov(Y,a^{-n}$),
which implies $\cov(Y,a^{-2n})\leq\cov(Y,a^{-n})^{2}$, and by induction,
$\cov(Y,a^{-kn})\leq\cov(Y,a^{-n})^{k}$. For $n>m$ we get 
\[
\dim(Y,a^{-n})\leq\dim(Y,a^{-([n/m]+1)m})\leq(1+\frac{m}{n})\dim(Y,a^{-m})
\]
The lemma follows (with constant $2$) in the case $r=a^{-m}$ and
$c=n/m$. For general $r,c$ approximate them by the nearest powers
of $a$, and estimating the error using $r<a^{-2}$.

A more detailed proof can be derived from Lemma \ref{lem:cr-bound-in-terms-of-HK}
below.
\end{proof}
\begin{lem}
\label{lem:full-entropy-implies-density}For every $2\leq b\in\mathbb{N}$
and $\varepsilon>0$ there exist $\lambda,r_{0}>0$ such that for
all $0<r<r_{0}$ and $\mu\in\mathcal{P}([0,1])$, 
\[
\frac{1}{\log(1/r)}H(\mu,r)>1-\lambda\qquad\Rightarrow\qquad\mu(\dense_{b}(\varepsilon))>1-\varepsilon.
\]
\end{lem}

\begin{proof}
Let $\varepsilon>0$ and set $\ell=\left\lceil \log_{b}(1/\varepsilon)\right\rceil +1$.
Write $B=\{0,\ldots,b-1\}$, let $x\in[0,1]$ and let $x_{1}x_{2},\ldots$
denote the base-$b$ digits of $x$. 

Fix $N\in\mathbb{N}$. If $x\in[0,1]\setminus\dense_{b}(\varepsilon,N)$,
then some block $w\in B^{\ell}$ is not a subword of $x_{1}\ldots x_{N}$.
Let us bound the number of sequences $x_{1}\ldots x_{N}\in B^{N}$
with this property: 
\begin{itemize}
\item There are $b^{\ell}-1$ choices for $w\in B^{\ell}$. 
\item Given $w\in B^{\ell}$, if $x_{1}\ldots x_{N}\in B^{N}$ does not
have $w$ as a subword, then it is a concatenations of $[N/\ell]$
sequences from $B^{\ell}\setminus\{a\}$ followed by at most $\ell$
additional digits. There are $(b^{\ell}-1)^{N/\ell}(1+o(1))$ choices
of such sequences.
\end{itemize}
Thus, there is a constant $\delta>0$ depending only on $\ell$ (hence
on $\varepsilon$) such that, as $N\rightarrow\infty$, the number
of $x_{1}\ldots x_{N}\in B^{N}$ corresponding to $x\in[0,1]\setminus\dense_{b}(\varepsilon,N)$
is at most 
\[
b^{\ell}\cdot(b^{\ell}-1)^{N/\ell}(1+o(1))\leq b^{(1-\delta)N}
\]
It follows that every $\nu\in\mathcal{P}([0,1]\setminus\dense_{b}(\varepsilon,N))$
is supported on at most $b^{(1-\delta)N}$ intervals of length $b^{-N}$
(each corresponding to one sequence $x_{1}\ldots x_{N}$ as above),
so 
\[
\frac{1}{\log b}H(\nu,b^{-N})\leq(1-\delta)N.
\]
To complete the argument, note that every $\mu\in\mathcal{P}([0,1])$
can be written as a convex combination of a probability measure $\nu$
on $[0,1]\setminus\dense_{b}(\varepsilon,N)$ and $\eta$ on $\dense_{b}(\varepsilon,N)$.
At scale $b^{-N}$ we have seen that the entropy of the former is
at most $(1-\delta)N\log b$, while that of the latter is trivially
bounded by $N\log b$. Therefore there is an $N_{0}$ such that if
$N>N_{0}$ and $H(\mu,b^{-N})>(1-\varepsilon\delta/2)N\log b$, then
the contribution of $\nu$ to the entropy must be less than an $\varepsilon$-fraction
of the mass of $\mu$. This is what we wanted to show, using the dictionary
$\lambda=\varepsilon\delta/2$, $r_{0}=b^{-N_{0}}$, $r=b^{-N}$ and
a standard interpolation argument to deal with $r$ that are not of
the form $b^{-N}$.
\end{proof}
Let $\left\Vert \cdot\right\Vert _{TV}$ denotes the total variation
distance between measures. It is well known that for $\mu,\nu\in\mathcal{P}([-R,R]^{d})$,
\begin{equation}
\left|H(\mu,r)-H(\nu,r)\right|\leq\left\Vert \mu-\nu\right\Vert _{TV}\log(R/r)+H(\left\Vert \mu-\nu\right\Vert _{TV})\label{eq:total-variation-entropy-bound}
\end{equation}
where $H(\underline{t})=-t\log t-(1-t)\log(1-t)$. 

A measure $\theta\in\mathcal{P}(\mathbb{R}^{d})$ has exact dimension
$\delta$, and we write $\dim\theta=\delta$, if
\begin{equation}
\lim_{r\rightarrow0}\frac{\log\theta(B_{r}(x))}{\log r}=\delta\qquad\text{for \ensuremath{\theta}-a.e. \ensuremath{x}}\label{eq:pointwise-dim}
\end{equation}
Equivalently, the limit holds with $\mathcal{Q}_{r}(x)$ instead of
$B_{r}(x)$, where $\mathcal{Q}_{r}(x)$ is the $r$-cube containing
$x$ as defined in the definition of scale-$r$ entropy. 

When $\dim\theta=\delta$ we have $\frac{1}{\log(1/r)}H(\theta,r)\rightarrow\delta$
(essentially one integrates (\ref{eq:pointwise-dim})). Also, $\dim\theta=\delta$
passes from $\theta$ to any measure $\eta\ll\theta$, so the entropy
limit holds for $\eta$ too. The following lemma is a variant of this
for a converging sequence of measures.
\begin{lem}
\label{lem:irrational-rotation-along-positive-density-set}Let $\theta,\theta_{n},\eta_{n}\in\mathcal{P}([-R,R]^{d})$
and assume $\dim\theta=\delta$ exists. Suppose that $\theta_{n}\rightarrow\theta$
weak-{*} and that $\eta_{n}\ll\theta_{n}$ with $d\eta_{n}/d\theta_{n}\leq C$
for some constant $C$. Then, given $\rho>0$, for all sufficiently
small $r>0$, for all sufficiently large $N$, 
\[
\frac{1}{\log(1/r)}H(\eta_{n},r)>\delta-\rho
\]
\end{lem}

\begin{proof}
Fix $\theta,\theta_{n},\eta_{n},C$ as in the statement. We can assume
that $\theta(\partial\mathcal{Q}_{r}(x))=0$ for all $r,x$, since
a random translate of the partitions in question satisfy this, and
the effect on entropy is negligible as $r\rightarrow0$.

Since $\dim\theta=\delta$, by (\ref{eq:pointwise-dim}) there exist
a function $\varepsilon(r)\rightarrow0$ as $r\rightarrow0$ such
that 
\[
\mu\left\{ x\;\left|\;|\frac{\log\theta(\mathcal{Q}_{r}(x))}{\log(1/r)}-\delta|<\varepsilon(r)\right.\right\} >1-\varepsilon(r)
\]
Let $E_{r}$ denote the set in the last equation, and $F_{r}=\mathbb{R}^{d}\setminus E_{r}$,
so $\theta(F_{r})<\varepsilon(r)$. Note that $E_{r},F_{r}$ are unions
of elements of $\mathcal{Q}_{r}$, so $\theta(\partial E_{r})=\theta(\partial F_{r})=0$.

Fix $r>0$ which we shall take to be small later. All expressions
$o(1)$ below are understood as $N\rightarrow\infty$, and may depend
on all previous parameters.

Define $\nu_{n}$ to be $\eta_{n}$ conditioned on $E_{r}$, that
is, $\nu_{n}=\frac{1}{\eta_{n}(E_{r})}\eta_{n}|_{E_{r}}$. 

Since $\theta_{n}\rightarrow\theta$ we have $\theta_{n}(F_{r})=\theta(F_{r})+o(1)<\varepsilon(r)+o(1)$.
Hence $\eta_{n}(F_{r})<C\theta_{n}(F_{r})<C\varepsilon(r)+o(1)$.
It follows that
\begin{equation}
\left\Vert \eta_{n}-\nu_{n}\right\Vert _{TV}<2C\varepsilon(r)+o(1)\label{eq:TV-estimate-for-eta-n}
\end{equation}
Moreover, since $d\nu_{n}/d\eta_{n}\leq1/\eta_{n}(E_{r})$,
\[
\frac{d\nu_{n}}{d\theta_{n}}<\frac{d\nu_{n}}{d\eta_{n}}\cdot\frac{d\eta_{n}}{d\theta_{n}}<\frac{1}{\eta_{n}(E_{r})}\cdot C=\frac{C}{1-C\varepsilon(r)-o(1)}
\]
Assuming, as we may, that $r$ is small enough and $N$ large enough
that $C\varepsilon(r)+o(1)<1/2$, we conclude that $d\nu_{n}/d\theta_{n}\leq2C$.
Hence, using $\theta_{n}\rightarrow\theta$ again, for $x\in E_{r}$
we get 
\[
\nu_{n}(\mathcal{Q}_{r}(x))\leq2C\cdot\theta_{n}(\mathcal{Q}_{r}(x))=2C\cdot\theta(\mathcal{Q}_{r}(x))+o(1)<2Cr^{\delta-\varepsilon(r)}+o(1)
\]
so, since $\nu_{n}$ is supported on $E_{r}$, 
\begin{align}
H(\nu_{n},r) & =-\int_{E_{r}}\log\nu_{n}(\mathcal{Q}_{r}(x))\,d\nu_{n}(x)\nonumber \\
 & >(\delta-\varepsilon(r))\log(R/r)-2C+o(1)\label{eq:entropy-estimate-for-nu-n}
\end{align}
Applying (\ref{eq:total-variation-entropy-bound}) to $\eta_{n},\nu_{n}$
and using (\ref{eq:TV-estimate-for-eta-n}) and (\ref{eq:entropy-estimate-for-nu-n}),
we get 
\begin{align*}
\frac{1}{\log(1/r)}H(\eta_{n},r) & >\delta-\varepsilon(r)-(2C\varepsilon(r)+o(1))-\frac{H(2C\varepsilon(r)+o(1))}{\log(1/r)}
\end{align*}
When $r$ is small enough, for all $N$ large this becomes $>\delta-\rho$,
as claimed.
\end{proof}

\subsection{\label{subsec:Proof-of-1-d}Proof of Theorem \ref{thm:main-1-d}}

Let $2\leq a,b\in\mathbb{N}$ be multiplicatively independent integers
and let $\mu\in\mathcal{P}([0,1])$ be an ergodic and non-atomic $a$-measure.

Consider the function $\Delta_{b}:[0,1]\rightarrow[0,1]$ given by
\[
\Delta_{b}(x)=\dim\overline{O_{b}(x)}.
\]
Since $T_{a}$ commutes with $T_{b}$ we have $T_{a}(\overline{O_{b}(x)})=\overline{O_{b}(T_{a}x)}$,
and $T_{a}$ preserves dimension since it is finite-to-one and locally
bi-Lipschitz. It follows that $\Delta_{b}$ is $T_{a}$-invariant,
and it is measurable, so by ergodicity of $\mu$, it is $\mu$-a.s.~constant.
Let 
\[
\delta=\;\text{the \ensuremath{\mu}-almost sure value of }\Delta_{b}.
\]

We prove the theorem in two steps: First, we show that it holds when
$\delta>0$, and then rule out the possibility $\delta=0$.

\subsection*{Step 1: $\delta>0$ implies that $\mu$-a.e.~orbit is $T_{b}$-dense}

We assume that $\delta>0$ and show that $\mu$-a.e.~$x$ has dense
$b$-orbit. The proof uses Host's theorem, and does not require a
spectral assumption on the dynamics of $\mu$. 

\subsubsection*{Choosing typical $x\in[0,1]$ and its $b$-orbit closure $X$}

Let $x\in[0,1]$ be $\mu$-typical in the sense that $\dim\overline{O_{b}(x)}=\delta$
and that $x$ visits $\dense(\varepsilon)$ with asymptotic frequency
$\mu(\dense(\varepsilon))$ for all $\varepsilon>0$. Set $X=\overline{O_{b}(x)}$.

\subsubsection*{Choosing $y\in X$ with uniformly distributed $a$-orbit}

Fix $y\in X$ whose $T_{a}$-orbit is uniformly distributed. To see
that such $y$ exists, note that $\dim X=\delta>0$ and $X$ is $T_{b}$
invariant, so its topological entropy is positive. Apply the variational
principle to find a $T_{b}$-invariant and ergodic probability measure
$\mu\in P(X)$ with positive entropy. By Host's theorem, $\mu$-a.e.~$y$
uniformly distributes under $T_{a}$.

\subsubsection*{Choosing $\varepsilon,k$ and approximating $y$ by $y'=T_{b}^{n}x\in O_{b}(x)$}

Let $\varepsilon>0$ and choose $k$ such that $\dense_{b}(\varepsilon,k)$
has Lebesgue measure $>1-\varepsilon$. Such a $k$ exists by ergodicity
of Lebesgue measure under $T_{b}$.

Since $\dense_{b}(\varepsilon,k)$ is an open set of measure $>1-\varepsilon$
and $y$ uniformly distributes under $T_{a}$, for all large enough
$N$, all but an $\varepsilon$ fraction of the elements of $(T_{a}^{n}y)_{n=1}^{N}$
lie in $\dense_{b}(\varepsilon,k)$. Fix such an $N$.

Since $\dense_{b}(\varepsilon,k)$ is open, the last conclusion also
for any $y'$ sufficiently close to $y$.

Since $y\in X=\overline{O_{b}(x)}$, there exists $n=n(N)$ with $y'=T_{b}^{n}x$
close enough to $y$ for the above to hold. Thus, $\{T_{a}^{i}y'\}_{i=1}^{N}$
visits $\dense_{b}(\varepsilon,k)$ at least $(1-\varepsilon)N$ times.

\subsubsection*{Concluding that $T_{a}^{i}x\in\dense_{b}(\varepsilon)$ for most
$1\protect\leq i\protect\leq N$, and finishing the proof}

By the identity $T_{a}^{i}y'=T_{b}^{n}(T_{a}^{i}x)$, every time that
$T_{a}^{i}y'\in\dense_{b}(\varepsilon,k)$ we have $T_{a}^{i}x\in\dense_{b}(\varepsilon,k+i)\subseteq\dense_{b}(\varepsilon)$,
and this holds for all but an $\varepsilon$-fraction of $1\leq i\leq N$.

Since the last conclusion holds for all $N$ sufficiently large, by
typicality of $x$ we conclude that $\mu(\dense_{b}(\varepsilon))>1-\varepsilon$.
This holds for every $\varepsilon>0$, which implies that the $T_{b}$-orbit
of $\mu$-a.e.~point is dense.
\begin{rem*}
For this proof to work, it would suffice to find $y\in X$ that uniformly
distributes along a subsequence, i.e.~$(1/N_{k})\sum_{i=1}^{N_{k}}\delta_{T_{a}^{i}y}\rightarrow Lebesgue$
for some sequence $N_{k}\rightarrow\infty$. To find such $y$ we
do not the full strength of Host's theorem, and could instead invoke
the Rudolph-Johnson theorem \cite{JohnsonRudolph95}, which asserts
that $(1/N)\sum_{n=1}^{N}T_{a}^{i}\nu\rightarrow Lebesgue$. By passing
a subsequence $N_{k}\rightarrow\infty$ we obtain uniform distribution
for $\nu$-a.e.~$y$ along $N_{k}$. We will use this observation
later for the proof of Theorem \ref{thm:main-multi-d}.
\end{rem*}

\subsection*{Step 2: Ruling out $\delta=0$ }

Assume $\delta=0$. We derive a contradiction by showing that $\mu$-a.e.~$x$
has dense $T_{b}$-orbit.

To this end fix $\varepsilon>0$. It suffices for us to show that
$\mu(\dense(\varepsilon))>1-2\varepsilon$.

Let $\lambda,r_{0}>0$ be the result of applying Lemma \ref{lem:full-entropy-implies-density}
with the given $\varepsilon$. Thus, a measure with entropy $>1-\lambda$
as scale $r_{0}$ is supported on many points with $\varepsilon$-dense
$a$-orbit.

\subsubsection*{Choosing a scale $r_{1}$ at which $\dim(\overline{O_{b}(x)},r_{1})\approx0$ }

Since $\dim_{B}\overline{O_{b}(x)}=\dim\overline{O_{b}(x)}=\delta=0$
for $\mu$-a.e.~$x$ we have 
\[
\lim_{r\rightarrow0}\dim(\overline{O_{b}(x)},r)=0\qquad\text{for }\mu\text{-a.e.~}x.
\]
Therefore, for all sufficiently small $r$ we have
\begin{equation}
\mu(x\mid\dim(\overline{O_{b}(x)},r)<\frac{1}{5}\lambda)>1-\varepsilon.\label{eq:definitoiun-of-r}
\end{equation}
Choose $r_{1}$ such that (\ref{eq:definitoiun-of-r}) holds for all
$r<r_{1}$. We can assume that $r_{1}<r_{0}$, and that $\frac{C}{\log(1/r_{1})}<\lambda$
where $C$ is the implicit constant in (\ref{eq:covering-bounds-entropy}). 

\subsubsection*{Choosing $D_{1},\ldots,D_{L}\subseteq[0,1]$ on which $x\protect\mapsto\dim(\overline{O_{b}(x)},r)$
is nearly\,constant }

Let $\mathcal{H}=\mathcal{H}([0,1])$ denote the space of closed non-empty
subsets of $[0,1]$, endowed with the Hausdorff measure. This is a
compact space.

Fix $\sigma>0$ small enough that if $X,Y\in\mathcal{H}$ with $d_{\mathcal{H}}(X,Y)<\sigma$,
and if $X\in\mathcal{H}$ satisfies $\dim(X,r_{1})<\frac{1}{5}\lambda$,
then $\dim(Y,r_{1})<\frac{1}{4}\lambda$, and in particular, by Lemma
\ref{lem:small-dimension-propogates-for-a-sets}, if $Y$ is also
a $b$-set, then $\dim(Y,r)<\frac{1}{3}\lambda$ for all $r\leq r_{2}$,
where $r_{2}<r_{1}$ is any small enough scale.

Partition $[0,1]$ into measurable sets $D_{i}$ such that $\{\overline{O_{b}(x)}\}_{x\in D_{i}}$
has diameter $<\sigma$ in $\mathcal{H}$. By compactness of $\mathcal{H}$,
we can assume this partition is finite. Write 
\[
Y_{i}=\overline{\bigcup_{x\in D_{i}}O_{b}(x)}.
\]
Note that this is a $b$-set.

By choice of $\sigma$, if there exists $x\in D_{i}$ with $\dim(\overline{O_{b}(x)},r_{1})<\frac{1}{8}\lambda$,
then $\dim(Y_{i},r)<\frac{1}{3}\lambda$ for all $r<r_{2}$, so by
(\ref{eq:definitoiun-of-r}),
\begin{equation}
\mu(\cup\{D_{i}\mid\dim(Y_{i},r)<\frac{1}{3}\lambda\}\text{ for all \ensuremath{r\leq r_{2}}})>1-\varepsilon.\label{eq:Yi-has-small-r-dim}
\end{equation}
Renumber the positive-measure sets $D_{i}$ in the event above as
$D_{1},\ldots,D_{L}$ for some $L$. Thus $\mu(D_{i})>0$ and 
\begin{equation}
\mu(\bigcup_{i=1}^{L}D_{i})>1-\varepsilon.\label{eq:measure-of-union-of-Ai}
\end{equation}
Let 
\[
C=\max_{1\leq i\leq L}\frac{2}{\mu(D_{i})}
\]

\subsubsection*{A sequence $t_{n}$ driven by an irrational rotation and its entropy
at scale $r$}

Define
\[
\alpha=\frac{\log a}{\log b}
\]
Write $\{s\}$ for the fractional part of $s$ and set $t_{n}=b^{\{\alpha n\}}$.
By multiplicative independence of $a,b$ we have $\alpha\notin\mathbb{Q}$,
so $\{\alpha n\}$ equidistributes for Lebesgue measure $\tau$ on
$[0,1]$, hence $t_{n}$ equidistributes to the push-forward $\theta$
of $\tau$ by $t\mapsto b^{t}$. Clearly $\theta$ is absolutely continuous
with bounded Radon-Nikodym derivative.

Apply Lemma \ref{lem:irrational-rotation-along-positive-density-set}
with $\rho=\varepsilon$, and with $(t_{n})$, $C$ as above. Let
$r_{3}<r_{2}$ be such that the conclusion of the lemma holds for
all $r<r_{3}$ and for $N$ correspondingly large. 

Fix such an $r<r_{3}$. 

\subsubsection*{Choosing a set $E\subseteq[0,1]$ of ``good'' generic points}

Choose a set $E\subseteq[0,1]$ of positive measure on which convergence
in the ergodic theorem holds uniformly for the indicator functions
of $D_{1},\ldots,D_{L}$ and $\dense_{b}(\varepsilon)$. Such a set
can be produced by repeated application of Egorov's theorem.

\subsubsection*{Choosing $x,y\in E$ with a very small distance $\approx b^{-N}$
between them}

Choose a $\mu$-typical $x\in E$. We can assume that $x$ is not
isolated in $E$, because $\mu(E)>0$ and $\mu$ is non-atomic.

Fix $y\in E\setminus\{x\}$ very close to $x$ and let 
\[
N=[\log_{b}\frac{1}{2|x-y|}].
\]
so that $|T_{b}^{n}x-T_{b}^{n}y|\leq1/2$ for all $0\leq n\leq N$
and $|T_{b}^{N}x-T_{b}^{N}y|\geq\frac{1}{2b}$.

Note that $N$ depends on $y$ and can be made arbitrarily large by
making $|y-x|$ small.

\subsubsection*{Expressing $y_{n}-x_{n}$ in terms of $t_{n}$}

For $0\leq n\leq[N/\alpha]$, consider the points
\begin{align*}
x_{n} & =T_{b}^{N-[\alpha n]}T_{a}^{n}x\\
y_{n} & =T_{b}^{N-[\alpha n]}T_{a}^{n}y.
\end{align*}
When we identify $y-x$ (and $y_{n}-x_{n}$) with a real number in
$[0,1]$, we get
\begin{align}
|y_{n}-x_{n}| & =b^{N}|y-x|\cdot b^{\alpha n-[\alpha n]}\nonumber \\
 & =c_{x,y}\cdot t_{n}\label{eq:yn-xn-difference}
\end{align}
where $t_{n}=b^{\{\alpha n\}}$, as above and $c_{x,y}\in[\frac{1}{2b},\frac{1}{2}]$.
The identity (\ref{eq:yn-xn-difference}) is true because the distance
between $x_{n},y_{n}$ never exceeds $1/2$ for $0\leq n\leq[N/\alpha]$,
hence there is no need to reduce modulo one. 

\subsubsection*{The entropy of $y_{n}$ at times that $T_{a}^{n}x\in D_{i}$ }

Fix $1\leq i\leq L$ and let 
\[
I_{i}=\{1\leq n\leq[N/\alpha]\mid T_{a}^{n}x\in D_{i}\}.
\]
By our choice of $E$ we have 
\begin{equation}
\frac{1}{[N/\alpha]}|I_{i}|=\mu(D_{i})+o(1)\label{eq:Ii-size}
\end{equation}
as $N\rightarrow\infty$. For large $N$ this is $>1/C$.

Fix a large $N$ as above and let 
\[
\theta_{i}=\frac{1}{|I_{i}|}\sum_{n\in I_{i}}\delta_{t_{n}}
\]
and
\[
\nu_{i}=\frac{1}{|I_{i}|}\sum_{n\in I_{i}}\delta_{y_{n}-x_{n}}=f\theta_{i}
\]
where $f(z)=c_{x,y}z$. Observe that all $t_{n}$ are distinct, so
by (\ref{eq:Ii-size}), writing $\theta'_{N}=\frac{1}{[N/\alpha]}\sum_{1\leq n\leq[N/\alpha]}\delta_{t_{n}}$,
we have $d\theta_{i}/d\theta'_{N}<C$ if $N$ is large enough (i.e.~when
$y$ is close enough to $x$). Thus, since $\theta'_{N}\rightarrow\theta$,
by Lemma \ref{lem:irrational-rotation-along-positive-density-set},
\[
\frac{1}{\log(1/r)}H(\theta_{i},r)>1-\frac{1}{2}\lambda.
\]
and since $f$ is bi-Lipschitz with constant independent of $r,N$,
if $r$ is small then 
\[
\frac{1}{\log(1/r)}H(\nu_{i},r)>1-\frac{1}{2}\lambda.
\]

On the other hand, we have $x_{n}\in Y_{i}$ for $n\in I_{i}$ (because
$x_{n}=T_{b}^{N-[\alpha n]}T_{a}^{n}x\in O_{b}(T_{a}^{n}x)\subseteq Y_{i}$
when $T_{a}^{n}x\in D_{i}$ and the last condition is equivalent to
$n\in I_{i}$). Thus, the uniform distribution $\nu'_{i}$ on the
sequence $(x_{n})_{n\in I_{i}}$ is supported on $Y_{i}$. Since $\dim(Y_{i},r)<\frac{1}{3}\lambda$
(see (\ref{eq:Yi-has-small-r-dim}) and the subsequent paragraph),
and using (\ref{eq:covering-bounds-entropy}), we have
\[
\frac{1}{\log(1/r)}H(\nu'_{i},r)\leq\dim(Y_{i},r)+O(\frac{1}{\log(1/r)})<\frac{1}{2}\lambda
\]
by our choice of $r$. 

It follows from Lemma \ref{lem:irrational-rotation-along-positive-density-set}
that the uniform distribution $\nu''_{i}$ on the sequence $(y_{n})_{n\in I_{i}}$
satisfies 
\begin{equation}
\frac{1}{\log(1/r)}H(\nu''_{i},r)>\frac{1}{\log(1/r)}H(\nu_{i},r)-\frac{1}{\log(1/r)}H(\nu'_{i},r)>1-\lambda.\label{eq:nearly-maximal-entropy}
\end{equation}
This inequality is explained as follows: Let $\tau_{i}$ denote the
uniform distribution on the sequence pairs $((x_{n},y_{n}))_{n\in I_{i}}$.
Then $\nu'_{i},\nu''_{i}$ are the marginal measures on $\tau_{i}$,
hence $H(\tau_{i},r)\leq H(\nu'_{i},r)+H(\nu'',r)$, and $\nu_{i}$
is the image of $\tau_{i}$ under the Lipschitz map $(x,y)\mapsto y-x$,
so $H(\nu_{i},r)\leq H(\tau_{i},r)+O(1)$. Combining the last two
inequalities and dividing by $\log(1/r)$, and assuming $r$ is small,
we get (\ref{eq:nearly-maximal-entropy}).

\subsubsection*{Completion of the proof}

By Lemma \ref{lem:full-entropy-implies-density}, our choice of $\lambda,r_{0}$,
and (\ref{eq:nearly-maximal-entropy}), for all but $1-\varepsilon$
of the elements $n\in I_{i}$ we have $y_{n}\in\dense_{b}(\varepsilon)$. 

Since $y_{n}\in O_{b}(T_{a}^{n}y)$ we have $O_{b}(T_{a}^{n}y)\supseteq O_{b}(y_{n})$,
and hence $T_{a}^{n}y\in\dense_{b}(\varepsilon)$ for $1-\varepsilon$
of $n\in I_{i}$.

For large $N$, the sequence $(T_{a}^{n}y)_{n=1}^{N}$ visits $D_{i}$
with frequency arbitrarily close to $\mu(D_{i})$ (since by definition
of $E$, convergence is uniform on $E$ and $y\in E$).

Since $\sum_{i=1}^{L}\mu(D_{i})>1-\varepsilon$, we conclude that
$T_{a}^{n}y\in\dense_{b}(\varepsilon)$ for $1-\varepsilon$ of times
$1\leq n\leq N$.

Finally, again using the fact that convergence of the ergodic averages
of $1_{\dense_{b}(\varepsilon)}$ is uniform on $E$ and that $y\in E$,
we conclude that
\[
\mu(\dense_{b}(\varepsilon))>1-2\varepsilon.
\]
This is what we were out to show.

\section{\label{sec:Proof-of-multi-d}Proof of Theorem \ref{thm:main-multi-d}}

In this section we prove Theorem \ref{thm:main-multi-d}. Fix non
singular, commuting matrices $A,B\in M_{d}(\mathbb{Z})$ that act
totally irreducibly on $\mathbb{T}^{d}$, along with a $T_{A}$-invariant
and ergodic non-atomic measure $\mu\in\mathcal{P}(\mathbb{T}^{d})$. 

We give the proof under the assumption (1) of the theorem, i.e.~that
$B$ is expanding. The proof in the hyperbolic case is the similar
but using two-sided orbits.

The structure of the proof is similar to the one-dimensional case,
but the details are more involved. One difference is that, due to
the non-conformal nature of the dynamics, for part of the argument
we work directly with topological entropy, rather than dimension.
Some background and preliminaries are given in Section (\ref{subsec:Coverings-and-topological-entropy})
and (\ref{subsec:Substantial-measures-lead-to-large-orbits}). 

Another difference is that for nearby points $x,y\in\supp\mu$, the
orbit $T_{B}^{N}(x-y)=B^{N}(y-x)\bmod1$ can behave differently depending
on the mixture of eigendirections represented in $y-x$, and the relatively
simple expression for $a^{n}b^{-[\alpha n]}$ appearing in in (\ref{eq:yn-xn-difference})
becomes a matrix expression $A^{n}B^{[\alpha n]}$ for an $\alpha$
that depends on the point we apply the matrices to. We discuss the
behavior of such products in Section (\ref{subsec:Largeness-of-Euclidean-orbits}).

\subsection{\label{subsec:Algebraic-consequences-of-assumptions}Algebraic preliminaries }

We recall some standard consequences of our assumptions on $A,B$,
and introduce notation some that will be used later.
\begin{lem}
Commutation and total irreducibility of $A,B$ imply that they are
jointly diagonalizable over $\mathbb{C}$. 
\end{lem}

\begin{proof}
It suffices to show that $A$ has distinct eigenvalues. If it did
not then its minimal polynomial $q$ would properly divide its characteristic
polynomial $p$. Then $\ker q(A)$ is a rational subspace which is
both $A$ and $B$ invariant, since $AB=BA$. This contradicts total
irreducibility.
\end{proof}
Let $\mathbb{R}=\oplus U^{i}$ denote the splitting of $\mathbb{R}^{d}$
into $B$-irreducible (equivalently, $A$-irreducible) subspaces and
write $\pi^{i}:\mathbb{R}^{d}\rightarrow U^{i}$ for the associated
projection. We identify the $\mathbb{R}$-linear space $U^{i}$ with
the field $\mathbb{F}^{i}=\mathbb{R}$ if the associated eigenvalue
is real, or $\mathbb{F}^{i}=\mathbb{C}$ if it is complex, so $A,B$
act on each $U^{i}\cong\mathbb{F}^{i}$ by scalar multiplication by
a scalar in $\mathbb{F}^{i}$. We refer to the elements of $U^{i}$
as eigenvectors and for an eigenvector $u$ write $a_{i},b_{i}\in\mathbb{F}^{i}$
for the eigenvalues of $A,B$ on $U^{i}$, respectively. Fixing an
element $0\neq u^{i}\in U^{i}$ for each $i$, we can write every
$u\in\mathbb{R}^{d}$ uniquely as $u=\sum c_{i}u^{i}$ with $c_{i}\in\mathbb{F}^{i}$. 
\begin{lem}
Total irreducibility and absence of rank one factors of the action
implies that $a_{i},b_{i}$ are multiplicatively independent for each
of the irreducible spaces $U^{i}$. 
\end{lem}

\begin{proof}
If for some $i$ and $m,n$ we have $a_{i}^{m}=b_{i}^{n}$, then $A^{m}B^{-n}|_{U^{i}}$
acts as the identity, so $A^{m}B^{-n}$ has $1$ for an eigenvalue.
The corresponding eigenspace is invariant under $A,B$ (since they
commute with $A^{m}B^{-n}$), contradicting total irreducibility unless
this is the entire space, in which case $A,B$ would be virtually
cyclic.
\end{proof}

\subsection{\label{subsec:Coverings-and-topological-entropy}Review of topological
entropy }

Let $X$ be a compact metric space and $T:X\rightarrow X$ a continuous
map. For an open cover $\mathcal{V}$ of $X$, write 
\[
\cov(X,\mathcal{V})=\,\text{the minimal size of a collection \ensuremath{\mathcal{V}}'\ensuremath{\subseteq}\ensuremath{\mathcal{V}} that covers \ensuremath{X}}
\]
If $\mathcal{W}$ is another open cover then $\mathcal{V}\lor\mathcal{W}=\{V\cap W\mid V\in\mathcal{V},W\in\mathcal{W}\}$.
Set 
\[
\mathcal{V}_{N}=\bigvee_{n=0}^{N-1}T^{-n}\mathcal{V}.
\]
The topological entropy of $(X,T)$ is given by
\[
h_{top}(X,T)=\sup_{\mathcal{V}}\inf_{N\in\mathbb{N}}\frac{1}{N}\log\cov(X,\mathcal{V}_{N}).
\]
where $\mathcal{V}$ runs over all open covers. By subadditivity of
the sequence, one could equivalently replace the infimum with the
limit as $N\rightarrow\infty$.

The map $T$ is said to be expansive with constant $c>0$, if $\sup_{n\geq0}d_{X}(T^{n}x,T^{n}y)>c$
for every distinct $x,y\in X$. Note that this property passes to
subsystems. For such a system, if the elements of $\mathcal{V}$ have
diameter $<c$, then $\mathcal{V}$ realizes the supremum in the definition
of entropy:

\begin{equation}
h_{top}(X,T_{B})=\inf_{N\in\mathbb{N}}\frac{1}{N}\log\cov(X,\mathcal{U}_{N}).\label{eq:top-entropy}
\end{equation}

\subsection{Topological entropy and dimension for toral endomorphisms}

Recall the notation of Section \ref{subsec:Algebraic-consequences-of-assumptions}.

Each $u\in\mathbb{R}^{d}$ can be written as $u=\sum c_{i}u^{i}$.
Define a norm on $\mathbb{R}^{d}$ by $\left\Vert u\right\Vert =\max|c_{i}|$
(here $|c_{i}|$ is the absolute value in the field $\mathbb{F}^{i}\cong U^{i}$).
This induces a metric on the torus $\mathbb{T}^{d}=\mathbb{R}^{d}/\mathbb{Z}^{d}$
which we shall use from now on.

Let $b_{\min}\leq b_{\max}$ denote the minimal and maximal modulus
of eigenvalues of $B$. With respect to the norm defined above we
have $\left\Vert B\right\Vert =b_{\max}$ and $\left\Vert B^{-1}\right\Vert =1/b_{\min}$.
Set
\begin{align*}
\rho & =\frac{1}{2b_{\min}}=\frac{\left\Vert B^{-1}\right\Vert }{2}\\
c & =\frac{1}{2b_{\max}}=\frac{1}{2\left\Vert B\right\Vert }.
\end{align*}
so $c\leq\rho<1$. 
\begin{lem}
$T_{B}$ is expansive with constant $c$.
\end{lem}

\begin{proof}
By our choice of metric, if $x,y\in\mathbb{T}^{d}$ and $d(x,y)<c$,
then 
\[
b_{\min}d(x,y)\leq d(T_{B}x,T_{B}y)\leq b_{\max}d(x,y)\leq\frac{1}{2}
\]
Thus, applying $T_{B}$ repeatedly as long as $d(T_{B}^{n}x,T_{B}^{n}y)<c$,
the distance must eventually surpass $c$.
\end{proof}

Once and for all, fix a partition $\mathcal{U}$ of $\mathbb{T}^{d}$
into balls of diameter $<c$ (hence $<\rho$), and in particular,
$<c$. Write $\mathcal{U}_{K}=\bigvee_{k=0}^{K-1}T_{B}^{-k}\mathcal{U}$.
Suppressing $\mathcal{U}$ in the notation, define for any $X\subseteq\mathbb{T}^{d}$
\begin{equation}
H_{top}(X,K)=\frac{\log\cov(X,\mathcal{U}_{K})}{K}.\label{eq:def-of-HK}
\end{equation}
By expansivity and choice of $\mathcal{U}$, we have $H_{top}(X,K)\rightarrow h_{top}(X,T_{B})$
for any subsystem $X\subseteq\mathbb{T}^{d}$.

The relations between $H_{top}(X,K)$ and $\dim(X,r)$ is clarified
in the following lemmas.
\begin{lem}
Let $\mathcal{U}$ be a cover of $\mathbb{T}^{d}$ by balls of diameter
$<c$, as above. Then 
\begin{equation}
\diam U\leq\rho^{N}\qquad\qquad\text{for all U\ensuremath{\in\mathcal{U}_{N}}}.\label{eq:UN-diameter-bound}
\end{equation}
\end{lem}

\begin{proof}
It suffices to prove that every $W\in\mathcal{U}_{N}$ is contained
in an ellipsoid of diameter $<\rho^{N}$. First, observe that if $V\subseteq\mathbb{T}^{d}$
is an ellipsoid, then $T_{B}^{-1}V$ consists of $|\det B|$ congruent
ellipsoids of diameter $\rho\diam V$, and the distance between the
center of mass of any two of them is $1/\left\Vert B\right\Vert =2c$.
Now proceed by induction. $N=1$ is trivial from the definition of
$\mathcal{U}$. Suppose that the claim holds for $\mathcal{U}_{N-1}$.
Given $W\in\mathcal{U}_{N}$, write $W=U\cap T_{B}^{-1}V$ for some
$U\in\mathcal{U}$ and $V\in\mathcal{U}_{N-1}$. By choice of $\mathcal{U}$
we know that $U$ has of diameter $<c$, so it can intersect at most
one of the $c$-separated ellipsoids that constitute $T_{B}^{-1}V$.
Each of these has diameter $<\rho\diam V<\rho\cdot\rho^{N-1}=\rho^{N}$,
as claimed. 
\end{proof}
It follows that 
\begin{equation}
H_{top}(X,K)\geq\dim(X,\rho^{K})\label{eq:HK-bounds-c-delta-to-the-K}
\end{equation}
We emphasize, however, that in general the elements of $\mathcal{U}_{N}$
are exponentially far from being balls, and the ratio of the two sides
of (\ref{eq:HK-bounds-c-delta-to-the-K}) can be exponential in $K$.
Nevertheless, by (\ref{eq:HK-bounds-c-delta-to-the-K}), 
\[
h_{top}(X)=0\qquad\Longrightarrow\qquad\dim_{B}X=0.
\]
More quantitatively, we have
\begin{lem}
\label{lem:cr-bound-in-terms-of-HK}Let $\mathcal{U}$ be as above
and let $X$ be $T_{B}$-invariant. If $H_{top}(X,K)<h$ then for
every $r<\rho^{K/h}$, 
\[
\dim(X,r)\leq O_{\rho,d}(h).
\]
\end{lem}

\begin{proof}
Using the relations $\cov(X,\mathcal{V}\lor\mathcal{W}$)$\leq\cov(X,\mathcal{V})\cdot\cov(X,\mathcal{W})$
and $\cov(X,T_{B}^{-1}\mathcal{V})=\cov(X,\mathcal{V})$, we see that
for every $m\in\mathbb{N}$, 
\begin{align*}
\cov(X,\mathcal{U}_{mK}) & =\cov(X,\bigvee_{n=0}^{mK-1}T_{B}^{-n}\mathcal{U})\\
 & \leq\prod_{\ell=1}^{m}\cov(X,\bigvee_{n=(\ell-1)K}^{\ell K-1}T_{B}^{-n}\mathcal{U})\\
 & =\cov(X,\mathcal{U}_{K})^{m}.
\end{align*}
Taking logarithms and dividing by $K$ we find that $H_{top}(X,mK)\leq H_{top}(X,K)$,
which, together with (\ref{eq:HK-bounds-c-delta-to-the-K}) yields
\[
\dim(X,\rho^{mK})\leq H_{top}(X,K)<h.
\]
This proves the claim when $r$ is an integer power of $\rho^{K}$.
For $\rho^{(m+1)K}<r<\rho^{mK}$ we interpolate in the standard way:
Use the bound we got for $\cov(X,\rho^{(m+1)K})$ to write
\[
\log\cov(X,r)\leq\log\cov(X,\rho^{(m+1)K})<h(m+1)K\log(1/\rho).
\]
Dividing by log(1/r) and then $r<\rho^{mK}$, we get
\[
\dim(X,r)=\frac{\log\cov(X,r)}{\log(1/r)}<(1+O(\frac{1}{m}))h.
\]
Since taking $r$ small amounts to taking $m$ large, this proves
the lemma.
\end{proof}

\subsection{\label{subsec:Substantial-measures-lead-to-large-orbits}Large $T_{B}$-orbits
from positive-entropy measures }

The analog of Lemma \ref{lem:full-entropy-implies-density} for $T_{B}$
says that if a measure on $\mathbb{T}^{d}$ have (normalized) entropy
near $d$ at small scales, then it is supported on points with nearly
dense $T_{B}$-orbits. This will not be useful for us, for in the
proof of Theorem \ref{thm:main-multi-d} we will want to apply this
to measures with positive, but far from maximal, entropy. The lemma
below provides an alternative.

We continue to fix the cover $\mathcal{U}$ from the previous section,
whose atoms are smaller than the expansivity constant $c$ of $T_{B}$. 

Given $K\in\mathbb{N}$ and \emph{$h>0$, }let\emph{
\begin{align}
\Omega_{K,h} & =\{x\in\mathbb{T}^{d}\mid H_{top}(\overline{O_{B}(x)},K)<h\}\nonumber \\
 & =\bigcup\{X\subseteq\mathbb{T}^{d}\mid X\text{ is a \ensuremath{T_{B}}-subsystem, and \ensuremath{H_{top}(X,K)<h}}\}.\label{eq:def-of-Omega-K-h}
\end{align}
}If $h$ is not too large, its complement has non-empty interior.
\begin{lem}
\label{lem:measure-on-Omega-K-h-have-small-entropy}For all $K\in\mathbb{N}$
and $\eta,h>0$ there exists $r_{0}>0$ such that, for all $r<r_{0}$,
if $\mu\in\mathcal{P}(\mathbb{T}^{d})$ satisfies $H(\mu,r)>\eta\log(1/r)$,
then $\mu(\Omega_{K,h})<1-\frac{\eta}{d}+\frac{3h}{d}$.
\end{lem}

\begin{rem*}
The constant $3$ in front of $h$ is arbitrary, any constant larger
than $1$ would suffice.
\end{rem*}
\begin{proof}
Fix $\eta,h,K$ and let $\mathcal{V}^{1},\ldots,\mathcal{V}^{L}$
be an enumeration of all subsets of $\mathcal{U}_{K}$ with fewer
than $2^{hK}$ elements. Note that $L$ depends on $K,h$ but we can
bound $L\leq2^{|\mathcal{U}_{K}|}$, independently of $h$.

If $X$ is a subsystem with $H_{top}(X,K)<h$, then by definition,
$X$ is covered by $\mathcal{V}^{i}$ for some $1\leq i\leq L$. Moreover,
it is covered by $T^{-Km}\mathcal{V}^{i}$ for all $m\geq0$, and
hence by $\bigvee_{m=0}^{M-1}T^{-Km}\mathcal{V}^{i}$ for all $M$.
The latter family is of cardinality $\leq|\mathcal{V}^{i}|^{M}<2^{KMh}$. 

Thus, for each $M$, the set $\Omega_{K,h}$ is covered by the union
of these families, i.e.~by
\[
\mathcal{W}_{M}=\bigcup_{i=1}^{L}\bigvee_{m=0}^{M-1}T^{-Km}\mathcal{V}^{i},
\]
and by the previous discussion,
\begin{equation}
|\mathcal{W}_{M}|\leq L\cdot2^{KMh}.\label{eq:WM-cardinality}
\end{equation}

Let $\rho$ be as in (\ref{eq:UN-diameter-bound}) and set 
\[
r_{0}=\rho^{K}
\]
for an $M$ which we shall determine later. By (\ref{eq:WM-cardinality}),
there exists $M_{0}\in\mathbb{N}$, depending on $L$ (hence on $K$,
but not on $\eta,h,$), such that for all $M>M_{0}$,
\begin{equation}
|\mathcal{W}_{M}|\leq r_{0}^{-2hM/M_{0}}\label{eq:WM-cardinality-in-terms-of-r0}
\end{equation}
At the same time, by \ref{eq:UN-diameter-bound} and the fact that
$\mathcal{W}_{M}\subseteq\mathcal{U}_{KM}$, every $W\in\mathcal{W}_{M}$
is a set of diameter $\leq\rho^{KM}=r_{0}^{M}$. 

Now suppose that $r<r_{0}$ and that $\mu\in\mathcal{P}(\mathbb{T}^{d})$
satisfies 
\begin{equation}
H(\mu,r)>\eta\log(1/r).\label{eq:mu-r-entropy}
\end{equation}
Choose $M$ so that $r=r_{0}^{M/M_{0}}$ (we ignore rounding errors,
which would lead to a $O(1/M_{0})$ term in the exponent, but do not
affect the argument). Let 
\[
E=\bigcup\{W\mid W\in\mathcal{W}_{M}\},
\]
so that $\Omega_{K,h}\subseteq E$, and write $\mu_{E}$, $\mu_{E^{c}}$
for the conditional measure of $\mu$ on $E$ and $E^{c}=\mathbb{T}^{d}\setminus E$
. Then by basic properties of entropy,
\begin{equation}
H(\mu,r)=\mu(E)H(\mu_{E},r)+(1-\mu(E))H(\mu_{E^{c}},r)+O(1)\label{eq:entropy-on-E-and-Ec}
\end{equation}
By (\ref{eq:WM-cardinality-in-terms-of-r0}), the measure $\mu_{E}$
is supported on $r^{-2h}$ sets of diameter $r$, giving 
\begin{equation}
H(\mu_{E},r)\leq2h\log(1/r).\label{eq:entropy-on-E-bound}
\end{equation}
On the other hand, we have the trivial bound 
\begin{equation}
H(\mu_{E^{c}},r)\leq d\log(1/r).\label{eq:entropy-on-Ec-bound}
\end{equation}
Combining (\ref{eq:mu-r-entropy}), (\ref{eq:entropy-on-E-and-Ec}),
(\ref{eq:entropy-on-E-bound}) and (\ref{eq:entropy-on-Ec-bound}),
dividing by $\log(1/r)$ we get 
\begin{align*}
\eta & \leq\frac{H(\mu,r)}{\log(1/r)}\\
 & \leq\frac{H(\mu_{E},r)}{\log(1/r)}+(1-\mu(E))\frac{H(\mu_{E^{c}},r)}{\log(1/r)}+O(\frac{1}{\log(1/r)})\\
 & \leq2h+(1-\mu(E))d+O(\frac{1}{\log(1/r)}).
\end{align*}
Assuming $M_{0}$ is large (and hence $r$ is small) relative to $h$
and rearranging, we obtain the desired bound on $\mu(E)$.
\end{proof}

\subsection{\label{subsec:Largeness-of-Euclidean-orbits}Distribution of $A^{n}B^{-[\alpha n]}u$ }

The theorem below is a special case of \cite[Section 4]{Hochman2025a}
with the roles of $A,B$ reversed. It is simpler because in \cite{Hochman2025a}
it is not assumed that $A,B$ commute. We provide the elementary proof
for completeness. The analysis is similar to that in Berend's work
\cite{Berend1983}.
\begin{thm}
\label{thm:Linear-orbits}Let $A,B,U^{i},a_{i},b_{i}$ be as in Section
\ref{subsec:Algebraic-consequences-of-assumptions}. Let $U=\oplus_{i\in I}U^{i}$
be some $B$-invariant (equiv. $A$-invariant) subspace that is expanded
by $B$, and set
\[
\alpha=\max\{\frac{\log|a_{i}|}{\log|b_{i}|}\mid i\in I\}
\]
(this is well defined since $|b_{i}|>1$ for $u\in U$, but $\alpha$
may be positive, zero or negative). Then there exist
\begin{itemize}
\item A torus $\mathbb{T}^{k}$,
\item An element $\gamma\in\mathbb{T}^{k}$ generating an infinite closed
group $\Gamma\leq\mathbb{T}^{k}$ endowed with Haar measure $\tau$,
\item A function $f:\mathbb{T}^{k}\rightarrow U$ satisfying
\begin{itemize}
\item [a.] $f$ is real-analytic on an open subset of $\mathbb{T}^{d}$
of full $\tau$-measure.
\item [b.] $\pi^{i}f|_{\Gamma}$ is non-constant if $U^{i}\leq U$ and
$\log a_{i}/\log b_{i}=\alpha$. Otherwise, $\pi^{i}f|_{\Gamma}=0$.
\end{itemize}
\end{itemize}
such that the following holds: If $u=\sum_{i\in I}c_{i}u^{i}\in U$
with $0\neq c_{i}\in\mathbb{F}^{i}$, and if we set $L=\diag(c_{i})\in\hom(U,U)$
and 
\[
u_{n}=A^{n}B^{-[\alpha n]}u,
\]
then
\begin{enumerate}
\item $u_{n}=L(f(n\gamma)+o(1))$.
\item \textup{Consequently, $(u_{n})$ equidistributes for $Lf\tau$. }
\end{enumerate}
\end{thm}

\begin{proof}
For every $n,\ell\in\mathbb{N}$, the matrix $A^{n}B^{-\ell}$ acts
on $U^{i}$ by multiplication by the scalar $a_{i}^{n}b_{i}^{-\ell}$.
Write $\varphi_{i,},\psi_{i}$ for the arguments of $a_{i},b_{i}$,
respectively (if $\mathbb{F}_{i}=\mathbb{R}$ then $\varphi_{i}=\psi_{i}=0$).
Writing $e(t)=e^{2\pi it}$, we have the identity 
\[
a_{i}^{n}b_{i}^{-[\alpha n]}=|a_{i}|^{n}|b_{i}|^{-[\alpha n]}\cdot e(\varphi_{i}n-\psi_{i}[\alpha n]).
\]
Writing $\{x\}$ for the fractional part of $x$ and noting that $[\alpha n]=\alpha n-\{\alpha n\}$,
the real and imaginary terms above become
\begin{align*}
|a_{i}|^{n}|b_{i}|^{-[\alpha n]} & =|a_{i}|^{(\frac{\log|d_{i}|}{\log|c_{i}|}-\alpha)n}\cdot|b_{i}|^{\{\alpha n\}}\\
e(\varphi_{i}n-\psi_{i}[\alpha n]) & =e(\varphi_{i}n-\psi_{i}\alpha n)\cdot e(\psi_{i}\{\alpha n\}).
\end{align*}
The behavior now depends on $i$. There are two cases:
\begin{description}
\item [{$\alpha>\log|a_{i}|/\log|b_{i}|$~:}] Then the real part is a
product of an exponentially decaying term and a bounded term. hence
$a_{i}^{n}b_{i}^{-[\alpha n]}\rightarrow0$ exponentially.
\item [{$\alpha=\log|a_{i}|/\log|b_{i}|$~:}] Then 
\[
a_{i}^{n}b_{i}^{-[\alpha n]}=|a_{i}|^{\{\alpha n\}}\cdot e(\varphi_{i}n-\psi_{i}\alpha n)\cdot e(\psi_{i}\{\alpha n\}).
\]
This is a piecewise smooth image of the orbit of $0\in\mathbb{T}^{2}$
under translation by 
\[
\gamma_{i}=(\alpha,\varphi_{i}-\alpha\psi_{i}).
\]
\end{description}
Putting this information together for all $i\in I$ and applying it
to $u=\sum_{i\in I}c_{i}u^{i}$, we obtain (1) and (2) with $\gamma=(\gamma_{i})_{i\in I}\in\mathbb{T}^{2|I|}$.

It remains only to verify that $\gamma$ is not a torsion element.
For this it suffices to show that when $i\in I$ is such that $\alpha=\frac{\log|a_{i}|}{\log|b_{i}|}$,
the numbers $\alpha,\varphi_{i}-\alpha\psi_{i}$ cannot both be rational.
Indeed, for such an $i$ suppose that $\alpha=p/q$ and $\varphi_{i}-\alpha\psi_{i}=p'/q$
for some integers $p,q,p',q'$. Then we would have
\[
a_{i}^{q}=e^{q\log|a_{i}|}\cdot e(q\varphi_{i})=e^{p\log|b_{i}|}e(q\alpha\psi_{i}+p')=b_{i}^{p}.
\]
which contradicts multiplicative independence of $a_{i},b_{i}$.
\end{proof}
Note that $\alpha,k,\gamma,f$ depend on the subspace $U$, and there
are finitely many choices for $U$, since it is a sum of $C$-irreducible
(and expanded) subspaces.

We will use the following property of the limiting distribution $\theta$
in the theorem, which should be seen as a generalization of Lemma
\ref{lem:irrational-rotation-along-positive-density-set}.
\begin{cor}
\label{cor:theta-have-substantial-entropy} The limiting distribution
$\theta=Lf\tau$ of the $u_{n}$ is a finite convex combination of
absolutely continuous measures on non-trivial smooth submanifolds
of $\mathbb{T}^{d}$. In particular the dimension $\dim\theta$ exists
pointwise (in the sense of (\ref{eq:pointwise-dim})) and is equal
to at least one. 
\end{cor}

\begin{proof}
In the notation of the theorem, the function $f$ is real-analytic
on an open set of full $\tau$ measure, so $f\tau$ decomposes as
absolutely continuous measures on finitely many submanifolds of $U$.
By the second stated property of $f$, the image of $f$ is contained
in the subspace $U'\leq U$ that is the sum of the $U^{i}\leq U$
for which $\pi^{i}u\neq0$. These correspond to the non-zero diagonal
entries of $L$, so $L$ is a linear automorphism of $U'$. Hence,
$\theta=Lf\tau$ also decomposes into absolutely continuous measures
on finitely many submanifolds. 
\end{proof}

\subsection{\label{subsec:Proof-of-Theorem-multi-d}Proof of Theorem \ref{thm:main-multi-d}}

Let $A,B\in M_{d}(\mathbb{Z})$ and $\mu$ be as described at the
start of Section (\ref{sec:Proof-of-multi-d}). As in the proof on
$\mathbb{R}/\mathbb{Z}$, set 
\[
O_{B}(x)=\{T_{B}^{n}x\}_{n=1}^{\infty}.
\]
Since $T_{A},T_{B}$ commute, the set $T_{A}\overline{O_{B}(x)}$
is again $T_{B}$-invariant, and since $T_{A}$ is a finite-to one
factor map from $\overline{O_{B}(x)}$ to its image, we have $h_{top}(\overline{O_{B}(x)})=h_{top}(T_{A}\overline{O_{B}(x)})$.
Using $\overline{O_{B}(T_{A}x)}=T_{A}\overline{O_{B}(x)}$, this topological
entropy is also equal to $h_{top}(\overline{O_{B}(T_{A}x)})$. It
now follows from ergodicity of $\mu$ that this quantity is $\mu$-a.e.~constant,
and we set 
\[
\delta=\;\text{the \ensuremath{\mu}-a.s.~value of \ensuremath{h_{top}(\overline{O_{B}(x)}.T_{B})}}.
\]
We again divide the proof in two, first proving the theorem under
the assumption that $\delta>0$, and then ruling out $\delta=0$.

\subsection*{Step 1: $\delta>0$ implies that $\mu$-a.e.~orbit is $T_{B}$-dense}

The proof is identical to Step 1 in Section \ref{subsec:Proof-of-1-d},
except for two changes. Let 
\[
\dense_{B}(\varepsilon,k)=\{x\in\mathbb{T}^{d}\mid\{T_{B}^{n}x\}_{n=1}^{k}\text{ is \ensuremath{\varepsilon}-dense}\}.
\]

\begin{description}
\item [{Finding~the~point~$y$}] : If $\delta=h_{top}\overline{O_{B}(x)}>0$,
then by the variational principle, we find a $T_{B}$-invariant and
ergodic measure $\nu$ on $\overline{O_{B}(x)}$ with positive entropy,
and apply the equidistribution theorem \cite[Theorem 1.1]{Hochman2025a}
to $\nu$ with the roles of $T_{A},T_{B}$ reversed, to conclude that
$\frac{1}{N}\sum_{n=1}^{N}T_{A}^{n}\nu\rightarrow Lebesgue$. Passing
to a subsequence of the Cesaro averages, as explained at the end of
Step 1, we find a point $y\in\overline{O_{B}(x)}$ that uniformly
distributes under $T_{A}$ for averages taken along some $N_{k}\rightarrow\infty$.
\item [{Lebesgue~measure~of~$\dense_{B}(\varepsilon,k)$}] : Total irreducibility
of the $A,B$-action implies that $B$ has no roots of unity, so Lebesgue
measure is ergodic for $T_{B}$. This implies that Lebesgue-a.e.~$x$
has dense $T_{B}$ orbit and hence for $k$ large the set of points
in $\dense_{B}(\varepsilon,k)$ has Lebesgue measure $>1-\varepsilon$.
\end{description}
The rest of the proof carries over to the multidimensional case.

\subsection*{Step 2: Ruling out $\delta=0$ }

Assume by way of contradiction that $\delta=0$. 

\subsubsection*{Choosing $K$ with $H_{top}(\overline{O_{b}(x)},K)\approx0$ (quantified
in terms of $h,\varepsilon$)}

Fix parameters
\[
0<h\ll\varepsilon<\frac{1}{4d}.
\]
As we will see later it suffices to take $h$ to be small multiple
of $\varepsilon$ (depending on $B$).

Let $H_{top}(\cdot,\cdot)$ be as defined in (\ref{eq:def-of-HK}).
Then 
\begin{equation}
\lim_{K\rightarrow\infty}H_{top}(\overline{O_{B}(x)},K)=h_{top}(\overline{O_{B}(x)})=\delta=0\qquad\qquad\text{for \ensuremath{\mu}-a.e.~\ensuremath{x}}\label{eq:zero-entropy-limit}
\end{equation}
 and we can choose $K\in\mathbb{N}$ so that 
\begin{eqnarray}
\mu(x & \in & \mathbb{T}^{d}\mid H_{top}(\overline{O_{B}(x)},K)<\frac{1}{2}h)>1-\frac{1}{3}\varepsilon.\label{eq:definition-of-K}
\end{eqnarray}

\subsubsection*{Choosing a scale $r_{0}$ at which large dimension implies large
$T_{B}$-orbits}

Let $r_{0}>0$ be the constant resulting from applying Lemma \ref{lem:measure-on-Omega-K-h-have-small-entropy}
to the parameters $K,h$ above and $\eta=1-2\varepsilon$.

Without loss of generality we can assume that $r_{0}<\rho^{K/h}$,
where $\rho$ is the constant in (\ref{eq:UN-diameter-bound}). Then,
by Lemma \ref{lem:cr-bound-in-terms-of-HK}, for every $r<r_{0}$
we have
\begin{equation}
H_{top}(\overline{O_{B}(x)},K)<h\qquad\Rightarrow\qquad\dim(\overline{O_{B}(x)},r)<O(h).\label{eq:small-Htop-K-implies-small-c-r}
\end{equation}
The implicit constant in the bound $O(h)$ depends only on $B$.

\subsubsection*{Choosing $D_{1},\ldots,D_{L}\subseteq\mathbb{T}^{d}$ on which $x\rightarrow H_{K}(\overline{O_{B}(x)})$
is nearly constant}

Let $\mathcal{H}=\mathcal{H}(\mathbb{T}^{d})$ denote the space of
closed non-empty subsets of $\mathbb{T}^{d}$, endowed with the Hausdorff
measure. This is a compact space.

Fix $\sigma>0$ small enough that if $X,Y\in\mathcal{H}$ satisfy
$d_{\mathcal{H}}(X,Y)<\sigma$, and if $H_{top}(X,K)<\frac{1}{2}h$,
then $H_{top}(Y,K)<h$.

Partition $\mathbb{T}^{d}$ into measurable sets $D_{i}$ such that
$\{\overline{O_{B}(x)}\}_{x\in D_{i}}$ has diameter $<\sigma$ in
$\mathcal{H}$. By compactness of $\mathcal{H}$, we can assume the
partition is finite. Write 
\begin{equation}
Y_{i}=\overline{\bigcup_{x\in D_{i}}O_{B}(x)}.\label{eq:def-of-Yi}
\end{equation}
By choice of $\sigma$, if there exists $x\in D_{i}$ such that $H_{top}(\overline{O_{B}(x)},K)<\frac{1}{2}h$,
then $H_{top}(Y_{i},K)<h$. Hence, by our choice of $K$ (i.e.~by
(\ref{eq:definition-of-K})),
\[
\mu(\cup\{D_{i}\mid H_{top}(Y_{i},K)<h\})>1-\varepsilon.
\]
Renumber the sets in the union above as $D_{1},\ldots,D_{L}$ for
some $L\in\mathbb{N}$. Thus 
\begin{equation}
\sum_{i=1}^{L}\mu(D_{i})=\mu(\bigcup_{i=1}^{L}D_{i})>1-\varepsilon.\label{eq:Ai-union-measure}
\end{equation}

Let $\mathcal{D}$ denote the partition into $D_{1},\ldots,D_{L}$
and the complement $D'=\mathbb{T}^{d}\setminus\bigcup_{i=1}^{L}D_{i}$.

\subsubsection*{Interlude: comparison with the case $d=1$}

At this point of the proof when $d=1$, we were ready to apply Lemma
\ref{lem:irrational-rotation-along-positive-density-set} to the sequence
$b^{\{\alpha n\}}$ along visit times of an $a$-orbit point one of
the sets $D_{i}$, and this determined the scale $r$ at which we
would work. Now, however, the sequence playing the role of $b^{\{\alpha n\}}$
is derived from a sequence $B^{n-[\alpha n]}A^{n}u$ as in Theorem
\ref{thm:Linear-orbits}, which in turn depends on a vector $u$ which
we have not determined yet. For this reason, we postpone choosing
the scale $r$ until later.

\subsubsection*{Choosing the set $E\subseteq\mathbb{T}^{d}$ of ``good'' $T_{A}$-generic
points}

Let $E\subseteq\mathbb{T}^{d}$ be a set with $\mu(E)>1-\varepsilon$
on which convergence in the ergodic theorem holds uniformly for the
indicator functions of $D_{1},\ldots,D_{L}$, and of $\Omega_{K,h}$
(as defined in (\ref{eq:def-of-Omega-K-h})). This can be done by
repeated application of Egorov's theorem.

\subsubsection*{Choosing $x$ and $y(k)$ with $u=\lim_{k\rightarrow\infty}B^{N(y_{k})}(y_{k}-x)\protect\neq0$}

Choose a $x\in E$ which we fix henceforth. We can assume $x$ is
not isolated in $E$, because $\mu$ is non-atomic and $\mu(E)>0$.

For any $y\in E$ with $d(x,y)<1/2$, we can identify $y-x$ with
a vector in $\mathbb{R}^{d}$ of length $<1/2$. We then define
\begin{equation}
N(y)=\max\{N\in\mathbb{N}\cup\{0\}\mid\left\Vert B^{N}(y-x)\right\Vert <1/2\}\label{eq:def-of-Ny}
\end{equation}
and set 
\begin{equation}
u_{y}=B^{N(y)}(y-x)\in\mathbb{R}^{d}.\label{eq:def-of-uy}
\end{equation}
Thus $1/(2\left\Vert B\right\Vert )\leq\left\Vert u_{y}\right\Vert \leq1/2$.

Since $y\in E$ can be chosen arbitrarily close to $x$ and $u_{y}$
are bounded away from $0,\infty$ in $\mathbb{R}^{d}$, we can choose
points $y(k)\in E\setminus\{x\}$ such that $y(k)\rightarrow x$ and
\[
u=\lim_{k\rightarrow\infty}u_{y(k)}
\]
exists. Note that $1/(2\left\Vert B\right\Vert )\leq\left\Vert u\right\Vert \leq1/2$,
and that $N(y(k))\rightarrow\infty$ since $y(k)\rightarrow x$.

\subsubsection*{Applying Theorem \ref{thm:Linear-orbits}: defining $U,\alpha,\gamma,\Gamma,\tau,f,L$
and $\theta$}

Let $U=\oplus\{U^{i}\mid\pi^{i}u\neq0\}$ denote the smallest $B$-invariant
subspace containing $u$, and let $\alpha,\gamma,\Gamma,\tau,f,L$
be associated to $U$ and $u$ as in Theorem \ref{thm:Linear-orbits}.
Let $\theta=Lf\tau$ denote the limiting distribution of $A^{n}B^{-[\alpha n]}u$.

.

\subsubsection*{The scale $r<r_{0}$ at which $\theta$ looks sufficiently smooth}

Let $\sigma=\frac{1}{2}\min\{\mu(D_{1}),\ldots,\mu(D_{L})\}$ and
apply Corollary \ref{cor:theta-have-substantial-entropy} to the measure
$\theta$ with this $\varepsilon$ and $\sigma$ as above. Note that
$\theta$ is exact dimension with a.e. dimension at least one, hence
we can choose $r<r_{0}$ and $N_{0}$ large with the following property:
For any $N>N_{0}$ and any set $I\subseteq\{1,\ldots,N\}$ of size
$>\sigma N$, for $u_{n}=A^{n}B^{-[\alpha n]}u\in\mathbb{R}^{d}$,
the measure $\theta_{I}=\frac{1}{|I|}\sum_{n\in I}\delta_{u_{n}}$
satisfies.
\begin{equation}
\frac{1}{\log(1/r)}H(\theta_{I},r)>1-\varepsilon\label{eq:theta-entropy}
\end{equation}

\subsubsection*{Defining $x_{n},y_{n}$ and comparing $y_{n}-x_{n}$ to $A^{n}B^{-[\alpha n]}u$}

Consider $y=y(k)$ and $N=N(y(k))$ for a large $k$ which we suppress
in our notation, and define the points
\begin{align*}
x_{n} & =T_{B}^{N-[\alpha n]}T_{A}^{n}x\\
y_{n} & =T_{B}^{N-[\alpha n]}T_{A}^{n}y.
\end{align*}
These are well defined in $\mathbb{T}^{d}$ as long as $N-[\alpha n]>0$,
hence for all $1\leq n\leq[N/|\alpha|]$ (in the case $\alpha=0$
we will just take $1\leq n\leq N$).

Lifting $y-x$ to a vector of norm $\leq1/2$ in $\mathbb{R}^{d}$,
we can write
\begin{align*}
y_{n}-x_{n} & =B^{N-[\alpha n]}A^{n}(y-x)\bmod1\\
 & =B^{-[\alpha n]}A^{n}(B^{N}(y-x))\bmod1\\
 & =A^{n}B^{-[\alpha n]}u_{y}\bmod1.
\end{align*}
As $k\rightarrow\infty$, the vector $u_{y}=u_{y(k)}$ becomes arbitrarily
close to $u$, and $N=N(y(k))$ becomes arbitrarily large. Thus we
can ensure that the points $y_{n}-x_{n}$ and $A^{n}B^{-[\alpha n]}u$
are arbitrarily close for $n$ in arbitrarily long segments $[1,M]$.
More precisely, we can choose a sequence $\rho(k)\rightarrow0$, and
integers $M(k)\ll N(y(k))$ with $M(k)\rightarrow\infty$ (and in
particular $M(k)>N_{0}$ for large $k$), such that
\begin{equation}
\sup_{1\leq n\leq M(k)}d(y_{n}-x_{n},A^{n}B^{-[\alpha n]}u)<\rho(k).\label{eq:yn-xn-approximate-un}
\end{equation}

\subsubsection*{The entropy of $y_{n}-x_{n}$ at times that $T_{A}^{n}x\in D_{i}$ }

For each $1\leq i\leq L$ define 
\[
I_{i}=\{1\leq n\leq M(k)\mid T_{A}^{n}x\in D_{i}\}.
\]
By our choice of $E$, and since $x\in E$ we have 
\[
\frac{1}{M(k)}|I_{i}|\rightarrow\mu(D_{i})\qquad\text{as }k\rightarrow\infty.
\]
In particular when $N$ (equivalently $k$) is large enough, $|I_{i}|>\sigma M(k)$.

\subsubsection*{Estimating the entropy of $(y_{n})_{n\in I_{i}}$}

For each $1\leq i\leq L$ we have arranged things so that Corollary
\ref{cor:theta-have-substantial-entropy} applies for all large $k$.
So, writing $\nu{}_{i}$ for the distribution $\nu{}_{i}$ of $(y_{n}-x_{n})_{n\in I_{i}}$,
for all sufficiently large $k$ we have 
\begin{equation}
\frac{1}{\log(1/r)}H(\nu_{i},r)>1-\varepsilon.\label{eq:nu-i-entropy-bound}
\end{equation}
On the other hand, by definition $x_{n}\in Y_{i}$ for $n\in I_{i}$,
so the uniform distribution $\nu'_{i}$ on the sequence $(x_{n})_{n\in I_{i}}$
is supported on $Y_{i}$. But $H_{top}(Y_{i},K)<h$, so by (\ref{eq:small-Htop-K-implies-small-c-r})
we have $\dim(Y_{i},r)=O(h)$. By (\ref{eq:covering-bounds-entropy}),
assuming as we may that $r$ is small enough, this implies 
\begin{equation}
\frac{1}{\log(1/r)}H(\nu'_{i},r)<O(h).\label{eq:nu-i-prime-entropy-bound}
\end{equation}
It follows from Lemma \ref{lem:irrational-rotation-along-positive-density-set}
that the uniform distribution $\nu''_{i}$ on the sequence $(y_{n})_{n\in I_{i}}$
satisfies 
\begin{align*}
\frac{1}{\log(1/r)}H(\nu''_{i},r) & >1-\varepsilon-O(h)\\
 & >1-2\varepsilon
\end{align*}
assuming, as we may, that $h\ll\varepsilon$. 

Hence, by our choice of $r_{0}$ (from Lemma \ref{lem:measure-on-Omega-K-h-have-small-entropy},
applied to $K,h$ and $\eta=1-2\varepsilon$), the definition (\ref{eq:def-of-Omega-K-h})
of $\Omega_{K,h}$, and since $r<r_{0}$, 
\begin{align}
\frac{1}{|I_{i}|}\{n\in I_{i}\mid y_{n}\notin\Omega_{K,h}\}= & \frac{1}{|I_{i}|}\{i\in I_{i}\mid H_{top}(\overline{O_{B}(y_{n})},K)\geq h\}\nonumber \\
 & >\frac{1-2\varepsilon}{d}-O(h)\\
 & >\frac{1}{2d}.\label{eq:times-when-y-n-is-not-in-Omega-K-h}
\end{align}
In the last line we used $\varepsilon<1/4$ and assumed, as we may,
that $h\ll\varepsilon$.

\subsubsection*{Completing the proof}

Since $y_{n}\in O_{B}(T_{A}^{n}y)$ we have $O_{B}(y_{n})\subseteq O_{B}(T_{A}^{n}y)$,
hence
\[
H_{top}(\overline{O_{B}(T_{A}^{n}y)},K)\geq H_{top}(\overline{O_{B}(y_{n})},K).
\]
Thus, by (\ref{eq:times-when-y-n-is-not-in-Omega-K-h}),
\[
\frac{1}{|I_{i}|}\{i\in I_{i}\mid T_{A}^{n}y\notin\Omega_{K,h}\}>\frac{1}{2d}.
\]
For large $k$, the sequence $(T_{A}^{n}y))_{n=1}^{M(k)}$ visits
$D_{i}$ with frequency $\mu(D_{i})+o(1)$ (since by definition of
$E$, convergence is uniform on $E$ and $y\in E$). Since $\sum_{i=1}^{L}\mu(D_{i})>1-\varepsilon>1/2$
(see (\ref{eq:Ai-union-measure}) and using $\varepsilon<1/4$), we
conclude that 
\begin{align*}
\frac{1}{M(k)}\{1\leq n\leq M(k)\mid T_{A}^{n}y\notin\Omega_{K,h}\} & \geq(1-\varepsilon)\frac{1}{2d}\geq\frac{1}{4d}.
\end{align*}

Since convergence of the ergodic averages of $1_{\Omega_{K,h}}$ is
uniform on $E$ and that $y\in E$, we conclude that by taking $k\rightarrow\infty$
that
\[
\mu(\Omega_{K,h})<1-\frac{1}{4d}.
\]
This contradicts (\ref{eq:definition-of-K}).

\bibliographystyle{plain}
\bibliography{bib}

\bigskip
\bigskip
\footnotesize
\noindent{}\texttt{Department of Mathematics, The Hebrew University of Jerusalem, Jerusalem, Israel\\ email: michael.hochman@mail.huji.ac.il}
\end{document}